\documentclass[12pt]{amsart}
\usepackage{epsfig}
\usepackage{mathtools}
\usepackage[subnum]{cases}
\usepackage{enumerate}
\usepackage{bbm,bm}
\usepackage{amsmath,amssymb,mathrsfs,amsthm}
\usepackage{color}
\usepackage{verbatim}
\usepackage{pgf,tikz,pgfplots,tikz-3dplot-circleofsphere,tkz-euclide}
\pgfplotsset{compat=1.14}
\usetikzlibrary{arrows, shapes.geometric}
\usepackage{graphicx}
\usepackage{ctable}
\graphicspath{ {./images/} }

\usepackage[hang]{footmisc} 

\newtheorem{theorem}{Theorem}[section]

\newtheorem{lemma}[theorem]{Lemma}
\newtheorem{corollary}[theorem]{Corollary}

\newtheorem{remark}[theorem]{Remark}

\newtheorem{conjecture}[theorem]{Conjecture}
\newtheorem*{property*}{Property}

\newcommand{\R}{\mathbb{R}}

\DeclareMathOperator{\dell}{del}
\DeclareMathOperator{\divv}{div}
\DeclareMathOperator*{\argmax}{argmax}
\DeclareMathOperator{\dist}{dist}
\DeclareMathOperator{\spann}{span}

\begin{document}
\setcounter{footnote}{0}

\title[Surface areas of equifacetal polytopes inscribed in the unit sphere $\mathbb{S}^2$]{Surface areas of equifacetal polytopes \\inscribed in the unit sphere $\mathbb{S}^2$}

    \author[N. Freeman, S. Hoehner, J. Ledford, D. Pack and B. Walters]{Nicolas Freeman, Steven Hoehner, Jeff Ledford, David Pack \\and Brandon Walters}
    \date{\today}

\maketitle

\begin{abstract}
This article is concerned with the problem of placing seven or eight points on the unit sphere $\mathbb{S}^2$ in $\mathbb{R}^3$ so that the surface area of the convex hull of the points is maximized. In each case, the solution is given for convex hulls with congruent isosceles or congruent equilateral triangular facets.
\end{abstract}

\section{Introduction and main results}

A classical problem in convex and discrete geometry asks: 
\begin{quote}
    \emph{Among all polytopes with $K$ vertices chosen from the unit sphere $\mathbb{S}^2$, which one has the greatest surface area?}
\end{quote}
\vspace{1mm}

\noindent This question is interesting in its own right, but also in view of its applications, which include geometric algorithms \cite{akkiraju}, crystallography \cite{DHL} and quantum theory \cite{kazakov-thesis}. Despite its simplicity, this question has turned out to be quite difficult to answer, and the complete solution has been given only in a handful of cases. Naturally, in all of the known cases, the maximizers exhibit the highest degree of symmetry possible. For example, among all convex polytopes inscribed in $\mathbb{S}^2$ with 4, 6 or 12 vertices, respectively, the maximum surface area polytope is the regular tetrahedron, octahedron, or icosahedron, respectively. 

In the table below, we list the known solutions of the surface area maximization problem for $K\geq 4$ points on the unit sphere $\mathbb{S}^2=\{(x,y,z)\in\R^3:\, x^2+y^2+z^2=1\}$. For an integer $K\geq 4$, let $\mathcal{I}_K$ denote the set of all convex polytopes with at most $K$ distinct vertices chosen from the unit sphere $\mathbb{S}^2$, and let $\mathcal{M}_K$ denote the corresponding subset of polytopes with congruent isosceles or congruent equilateral triangular facets. Also, $S(P)$ denotes the surface area of a polytope $P\in\mathcal{I}_K$. The second column lists the (global) surface area maximizers with $K$ vertices, and the third column gives the surface area of the  maximizer. 

\begin{center}
    \begin{tabular}{cccc}
    $K$   & $\argmax_{P\in\mathcal{I}_{K}}S(P)$ & $\max_{P\in\mathcal{I}_{K}}S(P)$ & Citation \\ \specialrule{.2em}{.1em}{.1em} 
    4  & regular tetrahedron & $8/\sqrt{3}\approx 4.62$  &  e.g., \cite{Toth-RegularFigures}\\ \hline
    5  & triangular bipyramid &  $3\sqrt{15}/2\approx 5.81$  & \cite{DHL}\\ \hline
    6  & regular octahedron   & $4\sqrt{3}\approx 6.93$   & e.g., \cite{Toth-RegularFigures}\\ \hline
    7  & --   & -- & --   \\ \hline
    8  &  -- & -- & -- \\\hline
    9  & -- & -- & -- \\ \hline
    10 & -- & -- & --\\ \hline
    11  &  -- & -- & --\\ \hline
    12 & regular icosahedron  & $2\sqrt{75}-2\sqrt{15}\approx 9.57$  & e.g., \cite{Toth-RegularFigures}\\ \hline
    $\geq 13$ & -- & -- & --
    \end{tabular}
    \end{center}
    
\vspace{5mm}
    
    Note that the global surface area maximizer in each of the known cases ($K=4,5,6,12$) is a member of $\mathcal{M}_K$. These maximizers are depicted in the figures below.  
    
    \begin{center}
    \def\r{1}
    \tdplotsetmaincoords{80}{90}
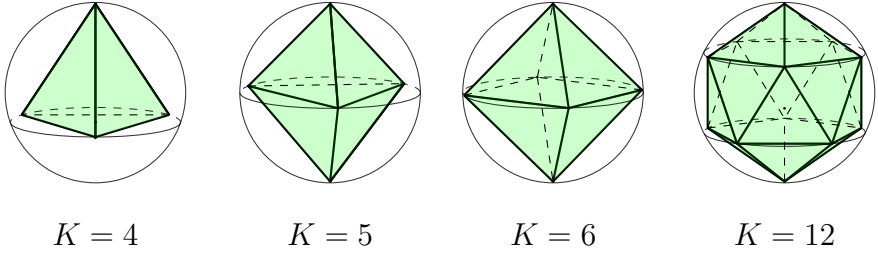
  \begin{tikzpicture}[scale=1.2,line join=bevel, tdplot_main_coords]
    \coordinate (O) at (0,0,0);

\coordinate (A) at ({sqrt(8/9)},0,{-1/3});
\coordinate (B) at ({-sqrt(2/9)},{sqrt(2/3)},{-1/3});
\coordinate (C) at ({-sqrt(2/9)},{-sqrt(2/3)},{-1/3});
\coordinate (D) at (0,0,1);

\begin{scope}[thick]
    \draw (A) -- (D)--(B);
    \draw (A) -- (D)--(C);
    \draw (A)--(B)--(D);
    \draw (A)--(C)--(D);
\end{scope}

\draw[thick,fill=green, opacity=0.2] (A) -- (D)--(B);
\draw[thick,fill=green, opacity=0.2] (A) -- (C) -- (D);

\begin{scope}[dashed] 
    \draw (B) -- (C);
\end{scope}

\begin{scope}[opacity=0.8]
\draw[tdplot_screen_coords] (0,0,0) circle (\r);
\tdplotCsDrawLatCircle{\r}{-19.4712206345}
\end{scope} 

\filldraw[black] (0,0,1) circle (0.25pt) node[anchor=south] {};
\filldraw[black] (0,0,{-1/3}) circle (0.25pt) node[anchor=north] {};
\filldraw[black] (1,0,{-1/3}) circle (0.25pt) node[anchor=north west] {};
 \coordinate (L1) at (0,0,-1);
    \node[yshift=-7mm] at (L1) {$K=4$};
\end{tikzpicture}
\tdplotsetmaincoords{80}{85}
\def\r{1}
  \begin{tikzpicture}[scale=1.2,line join=bevel, tdplot_main_coords]
    \coordinate (O) at (0,0,0);

\coordinate (A) at (1,0,0);
\coordinate (B) at ({-1/2},{sqrt(3)/2},0);
\coordinate (C) at ({-1/2},{-sqrt(3)/2},0);
\coordinate (D) at (0,0,1);
\coordinate (E) at (0,0,{-1});

\begin{scope}[thick]
    \draw (A) -- (D)--(B);
    \draw (A) -- (D) -- (C);
    \draw (A) -- (B)--(E);
    \draw (A)--(C)--(E);
    \draw (A)--(E)--(B);
\end{scope}

\draw[thick,fill=green, opacity=0.2] (A) -- (D)--(B);
\draw[thick,fill=green, opacity=0.2] (A) -- (D) -- (C);
\draw[thick,fill=green,opacity=0.2](A) -- (C) -- (E);  
\draw[thick,fill=green,opacity=0.2] (A)--(E)--(B);  

\begin{scope}[dashed] 
    \draw (C) -- (B);
\end{scope}

\begin{scope}[opacity=0.8]
\draw[tdplot_screen_coords] (0,0,0) circle (\r);
\tdplotCsDrawLatCircle{\r}{0}
\end{scope} 

\filldraw[black] (0,0,1) circle (0.25pt) node[anchor=south] {};
\filldraw[black] (0,0,-1) circle (0.25pt) node[anchor=north] {};
\filldraw[black] (1,0,0) circle (0.25pt) node[anchor=north west] {};
\filldraw[black] (B) circle (0.25pt) node[anchor=south] {};
\filldraw[black] (C) circle (0.25pt) node[anchor=south] {};
 \coordinate (L1) at (0,0,-1);
    \node[yshift=-7mm] at (L1) {$K=5$};
  \end{tikzpicture}
  \tdplotsetmaincoords{80}{80}
  \begin{tikzpicture}[scale=1.2,line join=bevel, tdplot_main_coords]
    \coordinate (O) at (0,0,0);

\coordinate (A) at (1,0,0);
\coordinate (B) at (0,1,0);
\coordinate (C) at ({-1},0,0);
\coordinate (D) at (0,{-1},0);
\coordinate (E) at (0,0,1);
\coordinate (F) at (0,0,{-1});
);

\begin{scope}[thick]
    \draw (B)--(E);
    \draw (B)--(F);
    \draw (D)--(E);
    \draw (D)--(F);
    \draw (A)--(E);
    \draw (A)--(F);
    \draw (B)--(A)--(D);
\end{scope}

\draw[thick,fill=green, opacity=0.2] (B) -- (A)--(E);
\draw[thick,fill=green, opacity=0.2] (A) -- (D) -- (E);
\draw[thick,fill=green,opacity=0.2](B) -- (A) -- (F);  
\draw[thick,fill=green,opacity=0.2] (A)--(D)--(F);  

\begin{scope}[dashed] 
    \draw (C) -- (E);
    \draw (C)--(F);
    \draw (D)--(C);
    \draw (B)--(C);
\end{scope}

\begin{scope}[opacity=0.8]
\draw[tdplot_screen_coords] (0,0,0) circle (\r);
\tdplotCsDrawLatCircle{\r}{0}
\end{scope} 

\filldraw[black] (0,0,1) circle (0.25pt) node[anchor=south] {};
\filldraw[black] (0,0,-1) circle (0.25pt) node[anchor=north] {};
\filldraw[black] (1,0,0) circle (0.25pt) node[anchor=north west] {};
\filldraw[black] (B) circle (0.25pt) node[anchor=north] {};
\filldraw[black] (C) circle (0.25pt) node[anchor=south] {};
\filldraw[black] (F) circle (0.25pt) node[anchor=south] {};
 \coordinate (L1) at (0,0,-1);
    \node[yshift=-7mm] at (L1) {$K=6$};
  \end{tikzpicture}
  \tdplotsetmaincoords{80}{90}
  \begin{tikzpicture}[scale=1.2,line join=bevel, tdplot_main_coords]
    \coordinate (O) at (0,0,0);

\coordinate (A) at (0,0,1);
\coordinate (B) at (0.894427191,0,0.4472135955);
\coordinate (C) at (0.2763932023,0.8506508084,{1/sqrt(5)});
\coordinate (D) at ({-0.7236067977},0.5257311121,{1/sqrt(5});
\coordinate (E) at ({-0.7236067977},{-0.5257311121},{1/sqrt(5)});
\coordinate (F) at (0.2763932022,{-0.8506508084},{1/sqrt(5)});
\coordinate (G) at (0.7236067977,0.5257311121,{-1/sqrt(5)});
\coordinate (H) at ({-0.2763932023},0.8506508084,{-1/sqrt(5)});
\coordinate (I) at ({-0.894427191},0,{-1/sqrt(5)});
\coordinate (J) at ({-0.2763932022},{-0.8506508084},{-1/sqrt(5)});
\coordinate (K) at (0.7236067977,{-0.5257311121},{-1/sqrt(5)});
\coordinate (L) at (0,0,-1);

\begin{scope}[thick]
    \draw (A) -- (F) -- (B);
    \draw (A) -- (B) -- (C);
    \draw (L) -- (K) -- (B);
    \draw (L) -- (G) -- (B);
    \draw (J) -- (F) -- (K);
    \draw (H) -- (C) -- (G);
    \draw (J) -- (K) -- (G) -- (H);
    \draw (A) -- (C);
    \draw (H) -- (L) -- (J);
\end{scope}

\draw[thick,fill=green,opacity=0.2] (A)--(F)--(B);
\draw[thick,fill=green,opacity=0.2] (A)--(C)--(B);
\draw[thick,fill=green,opacity=0.2] (J)--(K)--(F);  
\draw[thick,fill=green,opacity=0.2] (F)--(K)--(B);
\draw[thick,fill=green,opacity=0.2] (B)--(K)--(G);
\draw[thick,fill=green,opacity=0.2] (B)--(G)--(C);
\draw[thick,fill=green,opacity=0.2] (C)--(H)--(G);
\draw[thick,fill=green,opacity=0.2] (L)--(K)--(G);
\draw[thick,fill=green,opacity=0.2] (L)--(K)--(J);
\draw[thick,fill=green,opacity=0.2] (L)--(G)--(H);

\begin{scope}[dashed] 
    \draw (F) -- (E) -- (D) -- (C);
    \draw (E) -- (A) -- (D);
    \draw (J) -- (I);
    \draw (E) -- (I) -- (D);
    \draw (L) -- (I);
    \draw (I) -- (H);
    \draw (H) -- (D);
    \draw (J) -- (E);
\end{scope}

\begin{scope}[opacity=0.8]
\draw[tdplot_screen_coords] (0,0,0) circle (\r);
\tdplotCsDrawLatCircle{\r}{26.5650512}
\tdplotCsDrawLatCircle{\r}{-26.5650512}
\end{scope} 

\filldraw[black] (0,0,1) circle (0.25pt) node[anchor=south] {};
\filldraw[black] (0,0,-1) circle (0.25pt) node[anchor=north] {};
\filldraw[black] (1,0,0) circle (0.25pt) node[anchor=north west] {};
\filldraw[black] (B) circle (0.25pt) node[anchor=north] {};
\filldraw[black] (C) circle (0.25pt) node[anchor=south] {};
\filldraw[black] (F) circle (0.25pt) node[anchor=south]
{};
\filldraw[black] (G) circle (0.25pt) node[anchor=north]
{};
 \coordinate (L1) at (0,0,-1);
    \node[yshift=-7mm] at (L1) {$K=12$};
  \end{tikzpicture}
     \end{center}
    
    In our main results, we determine the surface area maximizers in $\mathcal{M}_7$ and $\mathcal{M}_8$.
    
    \begin{theorem}\label{7vertices}
        Let $P\in\mathcal{M}_7$. Then
$S(P) \leq \frac{5}{4}\sqrt{50-6\sqrt{5}}=7.560546\ldots$ 
with equality if and only if $P$ is a pentagonal bipyramid with two vertices at the poles $\pm e_3$ and the other five forming an equilateral pentagon in the equator $\mathbb{S}^2\cap (\spann(e_3))^\perp$.
    \end{theorem}
We conjecture that the pentagonal bipyramid stated in Theorem \ref{7vertices} is in fact the \emph{global} surface area maximizer in $\mathcal{I}_7$ (see also \cite[Sec. 7]{Hoehner-Ledford-2022}).
\begin{conjecture}
       Let $P\in\mathcal{I}_7$. Then
$S(P) \leq \frac{5}{4}\sqrt{50-6\sqrt{5}}=7.560546\ldots$ 
with equality if and only if $P$ is a pentagonal bipyramid with two vertices at the poles $\pm e_3$ and the other five forming an equilateral pentagon in the equator $\mathbb{S}^2\cap (\spann(e_3))^\perp$.
\end{conjecture}

Regarding the eight vertex problem, we have the following
    
    \begin{theorem}\label{mainThm}
 Let $P\in\mathcal{M}_8$. Then $S(P)\leq 8$ with equality if and only if $P$ has vertices at $\pm e_3$, $(\tfrac{\sqrt{8}}{3},0,\tfrac{1}{3}), (-\tfrac{\sqrt{2}}{3},\tfrac{\sqrt{6}}{3},\tfrac{1}{3}), (-\tfrac{\sqrt{2}}{3},-\tfrac{\sqrt{6}}{3},\tfrac{1}{3}), (\tfrac{\sqrt{2}}{3},\tfrac{\sqrt{6}}{3},-\tfrac{1}{3}), (-\tfrac{\sqrt{8}}{3},0,-\tfrac{1}{3})$ and $(\tfrac{\sqrt{2}}{3}, -\tfrac{\sqrt{6}}{3}, -\tfrac{1}{3})$.
    \end{theorem}
    
    \begin{corollary}
    The global surface area maximizer in $\mathcal{I}_8$ is not a member of  $\mathcal{M}_8$.
\end{corollary}

\begin{proof}
In \cite{Hoehner-Ledford-2022}, the coordinates of eight points on the sphere were given  which yield a polytope in $\mathcal{I}_8$ with surface area approximately equal to $8.11978$. By Theorem \ref{mainThm}, this is larger than the  surface area of any polytope in $\mathcal{M}_8$.
\end{proof}
\begin{center}
\tdplotsetmaincoords{80}{90}
\def\r{1}
  \begin{tikzpicture}[scale=2.1,line join=bevel, tdplot_main_coords]
    \coordinate (O) at (0,0,0);

\coordinate (A) at (1,0,0);
\coordinate (B) at ({(-1+sqrt(5))/4},{sqrt((5+sqrt(5))/8)},0);
\coordinate (C) at ({(-1-sqrt(5))/4},{sqrt((5-sqrt(5))/8)},0);
\coordinate (D) at (0,0,1);
\coordinate (E) at (0,0,{-1});
\coordinate (F) at ({(-1-sqrt(5))/4},{-sqrt((5-sqrt(5))/8)},0);
\coordinate (G) at ({(-1+sqrt(5))/4},-{sqrt((5+sqrt(5))/8)},0);

\begin{scope}[thick]
    \draw (A) -- (D)--(B);
    \draw (A) -- (B)--(E);
    \draw (G)--(A)--(E);
    \draw (D)--(G);
    \draw (G)--(E);
\end{scope}

\draw[thick,fill=blue, opacity=0.2] (A) -- (D)--(B);
\draw[thick,fill=blue, opacity=0.2] (A) -- (D) -- (G);
\draw[thick,fill=blue,opacity=0.2](A) -- (G) -- (E);  
\draw[thick,fill=blue,opacity=0.2] (A)--(E)--(B);  

\begin{scope}[dashed] 
    \draw (C) -- (B);
    \draw (D)--(C);
    \draw (D)--(F);
    \draw (C)--(F);
    \draw (F)--(G);
    \draw (E)--(F);
    \draw (E)--(C);
\end{scope}

\begin{scope}[opacity=0.8]
\draw[tdplot_screen_coords] (0,0,0) circle (\r);
\tdplotCsDrawLatCircle{\r}{0}
\end{scope} 

\filldraw[black] (0,0,0) circle (0.25pt) node[anchor=east] {\small $o$};
\filldraw[black] (0,0,1) circle (0.25pt) node[anchor=south] {$e_3$};
\filldraw[black] (0,0,-1) circle (0.25pt) node[anchor=north] {$-e_3$};
\filldraw[black] (1,0,0) circle (0.25pt) node[anchor=north west] {};
\filldraw[black] (B) circle (0.25pt) node[anchor=north] {};
\filldraw[black] (C) circle (0.25pt) node[anchor=south] {};
\filldraw[black] (F) circle (0.25pt) node[anchor=south]
{};
\filldraw[black] (G) circle (0.25pt) node[anchor=north]
{};
  \end{tikzpicture}
  \qquad
    \tdplotsetmaincoords{80}{90}
\def\r{1}
  \begin{tikzpicture}[scale=2.1,line join=bevel, tdplot_main_coords]
    \coordinate (O) at (0,0,0);

\coordinate (A1) at ({sqrt(8)/3},0,1/3);
\coordinate (B1) at ({(sqrt(8)/3)*cos(360/3)},{(sqrt(8)/3)*sin(360/3)},1/3);
\coordinate (C1) at ({(sqrt(8)/3)*cos(2*360/3)},{(sqrt(8)/3)*sin(2*360/3)},1/3);
\coordinate (D1) at (0,0,1);

\begin{scope}[thick]
    \draw (A1) -- (D1);
    \draw (D1)--(B1);
      \draw (D1)--(C1);
      \draw (B1)--(A1);
      \draw (A1)--(C1);
\end{scope} 

\begin{scope}[dashed] 
    \draw (C1) -- (B1);
\end{scope}

\coordinate (E1) at ({(sqrt(8)/3)*cos(60)},{(sqrt(8)/3)*sin(60)},-1/3);
\coordinate (F1) at ({(sqrt(8)/3)*cos(180)},{(sqrt(8)/3)*sin(180)},-1/3);
\coordinate (G1) at ({(sqrt(8)/3)*cos(300)},{(sqrt(8)/3)*sin(300)},-1/3);
\coordinate (H1) at
(0,0,-1);

\begin{scope}[thick]
        \draw (E1)--(G1);
        \draw (E1)--(H1);
        \draw (G1)--(H1);
\end{scope} 

\begin{scope}[dashed] 
    \draw (E1) -- (F1);
    \draw (F1)--(G1);
    \draw (F1)--(H1);
    \end{scope}

\begin{scope}[opacity=0.6]
\draw[tdplot_screen_coords] (0,0,0) circle (\r);
\tdplotCsDrawLatCircle{\r}{19.4712206345}
\tdplotCsDrawLatCircle{\r}{-19.4712206345}
\end{scope} 

\begin{scope}[thick]
        \draw (A1)--(E1); 
        \draw (B1)--(E1); 
        \draw (A1)--(G1);
        \draw (C1)--(G1);
\end{scope} 

\begin{scope}[dashed] 
    \draw (C1) -- (F1);
    \draw (B1)--(F1);
    \end{scope}
    
\draw[thick,fill=blue, opacity=0.2] (A1) -- (D1)--(B1);
\draw[thick,fill=blue, opacity=0.2] (A1) -- (D1)--(C1);
\draw[thick,fill=blue,opacity=0.2](A1) -- (G1) -- (E1);  
\draw[thick,fill=blue,opacity=0.2] (A1)--(E1)--(B1);  
\draw[thick,fill=blue, opacity=0.2] (A1) -- (G1)--(C1);
\draw[thick,fill=blue, opacity=0.2] (H1) -- (G1)--(E1);

\filldraw[black] (0,0,0) circle (0.25pt) node[anchor=east,xshift=0.5mm] {\small $o$};
\filldraw[black] (0,0,1) circle (0.25pt) node[anchor=south] {$e_3$};
\filldraw[black] (0,0,-1) circle (0.25pt) node[anchor=north] {$-e_3$};
\filldraw[black] (A1) circle (0.25pt) node[anchor=north west] {};
\filldraw[black] (B1) circle (0.25pt) node[anchor=north] {};
\filldraw[black] (C1) circle (0.25pt) node[anchor=south] {};
\filldraw[black] (G1) circle (0.25pt) node[anchor=south]
{};
\filldraw[black] (E1) circle (0.25pt) node[anchor=north]
{};
  \end{tikzpicture}
     \end{center}
     
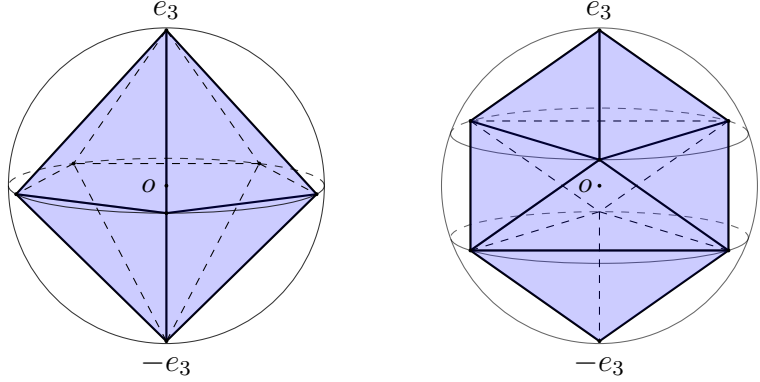
  {\flushleft\footnotesize {\bf Figure 1}: The maximum surface area polytope in $\mathcal{M}_7$ is depicted on the left, and  the maximum surface area polytope in $\mathcal{M}_8$ is depicted on the right.}
   
    
    \subsection{Comparison with the asymptotic best approximation}
    
    By Remark 2.1 in \cite{BHK},
\begin{equation*}
    \divv_2 (4\pi)^2 \leq \lim_{K\to\infty}K\min_{P\in\mathcal{I}_K}\left\{4\pi-S(P)\right\} \leq \dell_2(4\pi)^2
\end{equation*}
where $\divv_2=5/(18\sqrt{3})$ and $\dell_2=1/(2\sqrt{3})$ are the Dirichlet-Voronoi tiling number and Delone triangulation number in $\R^3$, respectively. See, for example, the works \cite{GruberIn, GruberOut,HK-DCG,zador1982asymptotic} for more background on these numbers, and asymptotic estimates for them as the dimension $n$ tends to infinity. Thus, up to an error of $O(K^{-2})$,
\begin{equation}\label{best-asymptotic}
4\pi\left(1-\frac{2\pi}{\sqrt{3}}K^{-1}\right) \lesssim \max_{P\in\mathcal{I}_K}S(P)  \lesssim 4\pi\left(1-\frac{10\pi}{9\sqrt{3}}K^{-1}\right)\quad \text{as} \quad K\to\infty,
\end{equation}
where $a_K\lesssim b_K$ means the sequence $(a_K)$ is less than or asymptotically equal to $(b_K)$. Let $a_K:=4\pi(1-\frac{2\pi}{\sqrt{3}}K^{-1})$ and $b_K:=4\pi(1-\frac{10\pi}{9\sqrt{3}}K^{-1})$. We compare the surface areas of the known maximizers with these asymptotic estimates in the  figure below.
    
    \begin{center}
 \begin{tikzpicture}
\begin{axis}[scale=1.3,legend style={at={(1,0.35)}, font=\footnotesize},
    axis lines = left,
    xlabel = \(K\),
    xtick={0,1,2,3,4,5,6,7,8,9,10,11,12,13},
    restrict y to domain=0:13,
    ytick={0,1,2,3,4,5,6,7,8,9,10,11,12,13},
    ymin=0, ymax=14, xmin=0.000001, xmax=14
    ]
scatter/classes={%
    a={mark=o,draw=black}}]
    \addplot[color=black,mark=*,only marks] coordinates {
	(4,4.61880215352)
    (5,5.80947501931)
    (6,6.92820323028)
    (12, 9.57454138327)
};
    \addplot[color=green,mark=*,only marks] coordinates {
        (7, 7.56054645584)
	(8,8)
};
    \addplot [domain=0:13, samples=2, dashed] {12.5663706144} node[below,pos=0.5]{{\scriptsize $\displaystyle 4\pi=\lim_{K\to\infty}\max_{P\in\mathcal{I}_{K}}S(P)$}};
    \addplot[domain=0:13,samples=200,blue] {4*pi-40*pi^2/(9*sqrt(3)*x)};
    \addplot[domain=0:13,samples=200,red] {4*pi-8*pi^2/(sqrt(3)*x)};
    \legend{$\max_{P\in\mathcal{I}_K}S(P)$, $\max_{P\in\mathcal{M}_K}S(P)$, $\displaystyle 4\pi$, $b_K$, $a_K$}
\end{axis}
\end{tikzpicture}
\end{center}
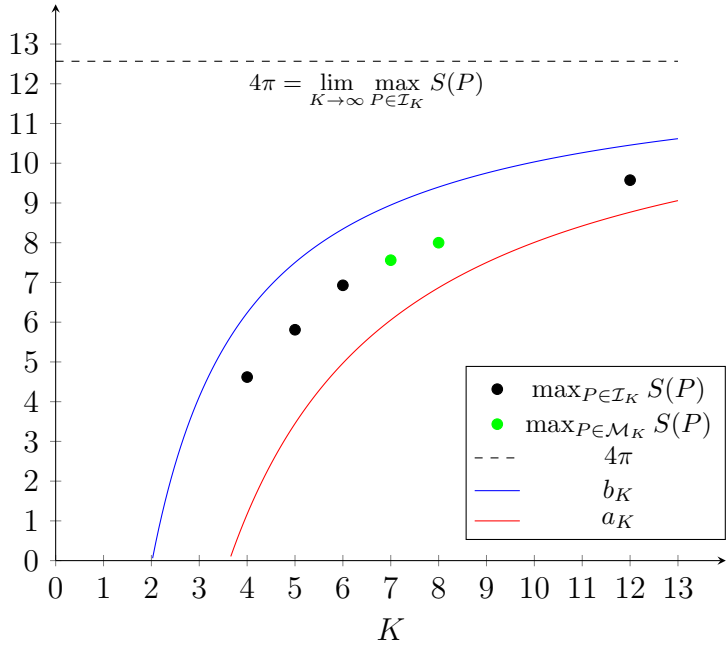
  {\flushleft\footnotesize {\bf Figure 2}: The surface areas of the known maximizers lie between the asymptotic estimates given in \eqref{best-asymptotic}.}

\section{Graph colorings of polytopes and geometric defects}

Our main strategy will be to rule out possibilities by leveraging the interplay between polytopes and colored graphs.  For a given polytope $P$ in $\R^3$, let $G(P)=(V(P), E(P))$ denote its graph, where $V(P)$ is the set of vertices of $P$ and $E(P)$ is the set of edges of $P$.  A polytope $P$ with congruent isosceles facets corresponds to a 2-coloring of $G(P)$.  The following geometric observation will be useful. 

\begin{lemma}\label{rule-out-lemma}
Let $K\geq 5$ be an integer and suppose that $P\in\mathcal{M}_K$ is a polytope with corresponding 2-colored graph $G(P)$.  If a pair of vertices has three distinct identically colored length two paths between them, then these vertices are antipodal.    
\end{lemma}

\begin{proof}
Fix a pair of vertices $v_1,v_2\in V(P)$ with the property described.  Let the  three identically colored 2-paths correspond to the vertices $u_1,u_2,u_3\in V(P)$, which lie in the plane $H$ (say).  The sphere meets $H$ in a circle which contains $u_1,u_2$ and $u_3$.  The line $\ell$ containing $v_1$ and $v_2$ is orthogonal to $H$.  Let the point $\pi_\ell\in H$ be the orthogonal projection of $\ell$ into $H$.

For $1\leq j\leq 3$, we have $\dist(\pi_\ell, u_j)=h$, where $h$ is the shared height of the triangles $\triangle[v_1,u_j,v_2]$.  This forces $\pi_\ell$ to be the center of the circle containing $u_1,u_2$ and $u_3$.  Since $\ell$ is orthogonal to $H$, the segment connecting $v_1$ to $v_2$ contains the origin, which means that these vertices are antipodal.
\end{proof}

We can say more if the pair of vertices share an edge.  In particular, we can rule out potential colorings with the following 

\begin{corollary}[Defect A]
Let $K\geq 5$ be an integer and let  $P\in\mathcal{M}_K$ have a 2-colored graph $G(P)$. Suppose that   $v_1,v_2\in V(P)$ satisfy the hypotheses of Lemma \ref{rule-out-lemma} and share an edge of a given color. If $v_1$ or $v_2$ is contained in a different edge of the same color, then $P$ is degenerate.
\end{corollary}

\begin{proof}
Since $v_1$ and $v_2$ are antipodal, their shared edge, colored red (say), has length 2.  Suppose that the other red edge is shared by $v_2$ and $v_3$.  This forces $v_3$ to be $v_1$.
\end{proof}

\begin{remark}
Following the proof, if there is a ``red" length two path connecting \emph{any} three vertices, then $P\notin \mathcal{M}_K$. 
\end{remark}

The next property can be used to eliminate nearly all colorings which feature a degree 3 vertex.

\begin{corollary}[Property $\mathscr{L}$; c.f.  \cite{DHL}]\label{Prop-L-cor}
Let $K\geq 5$ be an integer and suppose that $P\in\mathcal{M}_K$ has a 2-colored graph $G(P)$ and a degree three vertex $v\in V(P)$.  If the edges incident to $v$ are not all the same color,   then the coloring is degenerate.  
\end{corollary}

\begin{proof}
Such a coloring forces a version of Defect A.  Suppose that each isosceles facet is colored with two blue edges and one red edge.  Let $a,b,c\in V(P)$ be the vertices incident to $v$.  We can generate three blue paths of length 2 which connect two of the vertices that share an edge with $v$ (see the figure below; it will be the edge colored red).

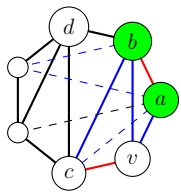
\begin{center}
\begin{tikzpicture}[scale=1]
    \coordinate (u1) at (1,0);
    \coordinate (u2) at ({cos(360/7)},{sin(360/7)});
    \coordinate (u3) at ({cos(2*360/7)},{sin(2*360/7)});
    \coordinate (u4) at ({cos(3*360/7)},{sin(3*360/7)});
    \coordinate (u5) at ({cos(4*360/7)},{sin(4*360/7)});
    \coordinate (u6) at ({cos(5*360/7)},{sin(5*360/7)});
    \coordinate (u7) at ({cos(6*360/7)},{sin(6*360/7)});
    
    \begin{scope}[thick]
    \draw[red] (u1)--(u2);
    \draw[black] (u2)--(u3);
    \draw[black] (u3)--(u4);
    \draw[black] (u4)--(u5);
    \draw[black] (u5)--(u6);
    \draw[red] (u6)--(u7);
    \draw[blue] (u7)--(u1);
    \end{scope}
    
        \begin{scope}[thick]
    \draw[blue] (u2)--(u7);
    \draw[blue] (u2)--(u6);
    \draw[black] (u3)--(u6);
    \draw[black] (u3)--(u5);
    \end{scope}
    
       \begin{scope}[dashed]
    \draw[blue] (u2)--(u4);
    \draw[blue] (u4)--(u1);
    \draw[black] (u5)--(u1);
    \draw[blue] (u6)--(u1);
    \end{scope}
    
     \node[circle,draw,fill=green,scale=0.7] (d1) at (u1){\color{black}$a$};
    \node[circle,draw,fill=green,scale=0.7] (d2) at (u2){\color{black}$b$};
    \node[circle,draw,fill=white,scale=0.7] (d3) at (u3){\color{black}$d$};
    \node[circle,draw,fill=white,scale=0.7] (d4) at (u4){\color{black}};
    \node[circle,draw,fill=white,scale=0.7] (d5) at (u5){\color{black}};
    \node[circle,draw,fill=white,scale=0.7] (d6) at (u6){\color{black}$c$};
    \node[circle,draw,fill=white,scale=0.7] (d7) at (u7){\color{black}$v$};
    
\end{tikzpicture}
\end{center}

\noindent This forces the aforementioned red edge to be a diameter, as in Defect A.  The vertex $v$ is the apex of a tetrahedron whose base is also an isosceles triangle with two blue edges and one red edge.  Now attempting to color one of the facets that shares a blue edge with the base of this tetrahedron yields a path of length two consisting of red edges.  
In the picture above, we may try to color $\triangle[b,c,d]$ with a red edge between $b$ and $d$, but this would force vertex $a$ to coincide with  vertex $d$.  Alternatively, we may try to color the edge joining $c$ and $d$ red, which would force the vertices $v$ and $d$ to coincide.  Either way, $P\notin \mathcal{M}_K$.  
\end{proof}

\begin{remark}
As we shall see, Corollary \ref{Prop-L-cor} can be used to  greatly reduce the amount of plausible colorings,  and allows us to use Property $\mathscr{L}$ from \cite{DHL}.
\end{remark}

The next result deals with potential colorings of $G(P)$ in the setting of Lemma \ref{rule-out-lemma}, but now assumes further that each of the paths between the vertices has the same color.

\begin{corollary}[Defect B]
Let $K\geq 5$ be an integer and let $P\in\mathcal{M}_K$ have 2-colored graph $G(P)$. If $v_1,v_2\in V(P)$ satisfy the hypotheses of Lemma \ref{rule-out-lemma}, and (at least) three of the paths are one color and the remaining paths are the alternate color, then the corresponding edge lengths are equal. 
\end{corollary}

\begin{proof}
Suppose that  three paths are colored blue and  correspond to vertices $u_j$ and the remaining  paths (colored red) correspond to the vertices $w_k$.  Each of the triangles $\triangle[v_1,u_j,v_2]$ and $\triangle[v_1,w_j,v_2]$ are isosceles. Hence the plane $H$ contains both sets of points $\{u_j\}$ and $\{w_k\}$ since it contains the midpoint of the segment joining $v_1$ and $v_2$.   
\end{proof}

\begin{remark}
Defect B is useful in that it forces equilateral facets.  By the well-known fact that the angle defect around any vertex of a convex polytope must be positive, we can thus rule out all polytopes which exhibit Defect B and contain a vertex of degree six or higher.
\end{remark}

There is one more special case that deserves mention.
\begin{corollary}[Defect C]
Let $K\geq 5$ be an integer and suppose that $P\in\mathcal{M}_K$ has a 2-colored graph $G(P)$.  Suppose that there is a vertex $v\in V(P)$ incident to four coplanar vertices, where the edges are colored in alternating fashion.  Then these edge lengths are the same.  
\end{corollary}

\begin{proof}
The plane containing the four vertices intersects the unit ball in a disc $D$.  Let the point $\pi_v$ be the orthogonal projection of $v$ into $D$.  Then we have two pairs of congruent right triangles in an alternating position.  Consider the legs of these triangles, of length $l_1$ and $l_2$, that lie in the disc emanating from $\pi_v$. Now construct the disc $D_1$ centered at $\pi_v$ of radius $l_1$.  If $\pi_v$ is not the center of $D$, then one of the legs of length $l_2$ is interior to $D_1$, which shows that $l_1>l_2$.  However, since the lengths are alternating, the other leg of length $l_2$ passes out of $D_1$ which shows that $l_1<l_2$.  But this is impossible; hence we must have that $\pi_v$ is the center of $D$, which means that $l_1=l_2$.  The result follows.   
\end{proof}

\begin{remark}
Defect C can occur, for example, if (as in Defect B) a 2-colored graph has four 1-color paths: two red and two blue. The isosceles triangles provide that the four incident vertices are coplanar; if the corresponding edge colors alternate, then Defect C is present.
\end{remark}

\section{Warm-up: The cases $K=4,5,6$}

To demonstrate the method of using Defects A, B and C to rule out combinatorial types, we show how they can be used to determine the surface area maximizers in $\mathcal{M}_K$ for $K=4,5,6$.

\subsection{$K=4$} 
In this case, Property $\mathscr{L}$ forces the tetrahedron to be regular.
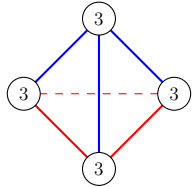
\begin{center}
\begin{tikzpicture}[scale=1]
    
    \coordinate (u1) at (1,0);
    \coordinate (u2) at ({cos(360/4)},{sin(360/4)});
    \coordinate (u3) at ({cos(2*360/4)},{sin(2*360/4)});
    \coordinate (u4) at ({cos(3*360/4)},{sin(3*360/4)});
    
    \begin{scope}[thick]
    \draw[blue] (u1)--(u2);
    \draw[blue] (u2)--(u3);
    \draw[red] (u3)--(u4);
    \draw[red] (u4)--(u1);
    \draw[blue] (u2)--(u4);
    \end{scope}
    
        \begin{scope}[thick]
    \end{scope}
    
       \begin{scope}[dashed]
    \draw[red] (u1)--(u3);
    \end{scope}
    
     \node[circle,draw,fill=white,scale=0.6] (d1) at (u1){3};
    \node[circle,draw,fill=white,scale=0.6] (d2) at (u2){3};
    \node[circle,draw,fill=white,scale=0.6] (d3) at (u3){3};
    \node[circle,draw,fill=white,scale=0.6] (d4) at (u4){3};

\end{tikzpicture}
\end{center}

\subsection{$K=5$} 
Property $\mathscr{L}$ forces the surface area maximizer in $\mathcal{M}_5$ to be a triangular bipyramid with an equilateral triangle connecting the degree 4 vertices, and the two degree 3 vertices must be antipodal by Lemma \ref{rule-out-lemma}.

\begin{center}
\begin{tikzpicture}[scale=1]
    
    \coordinate (u1) at (1,0);
    \coordinate (u2) at ({cos(360/5)},{sin(360/5)});
    \coordinate (u3) at ({cos(2*360/5)},{sin(2*360/5)});
    \coordinate (u4) at ({cos(3*360/5)},{sin(3*360/5)});
    \coordinate (u5) at ({cos(4*360/5)},{sin(4*360/5)});
    
    \begin{scope}[thick]
    \draw[blue] (u1)--(u2);
    \draw[blue] (u2)--(u3);
    \draw[blue] (u3)--(u4);
    \draw[red] (u4)--(u5);
    \draw[blue] (u5)--(u1);
    \end{scope}
    
        \begin{scope}[thick]
   \draw[red] (u2)--(u4);
    \draw[blue] (u4)--(u1);
    \end{scope}
    
       \begin{scope}[dashed]
    \draw[red] (u5)--(u2);
    \draw[blue] (u5)--(u3);
    \end{scope}
    
     \node[circle,draw,fill=green,scale=0.6] (d1) at (u1){3};
    \node[circle,draw,fill=white,scale=0.6] (d2) at (u2){4};
    \node[circle,draw,fill=green,scale=0.6] (d3) at (u3){3};
    \node[circle,draw,fill=white,scale=0.6] (d4) at (u4){4};
   \node[circle,draw,fill=white,scale=0.6] (d5) at (u5){4};


\end{tikzpicture}

\end{center}

\subsection{$K=6$} 
The possible 2-colored graphs that arise in this case are depicted below. The one with a degree 3 vertex is subject to Property $\mathscr{L}$, which subsequently forces $a=b$. However, this is impossible since it would force the two degree 5 vertices to be antipodal, which would mean that all of the edges have length 2, a contradiction.

There are two plausible 2-colorings for the graph of an inscribed octahedron. The first, labeled Class 2 (i),  has a cycle of 4 red edges. By Lemma \ref{rule-out-lemma}, this forces the configuration to be that of a regular octahedron. (One could also use Defect C here to get antipodal points in the red ring of vertices.)

The second 2-coloring, labeled Class 2 (ii), has two red paths of length 2. But then Lemma \ref{rule-out-lemma} forces a pair of antipodal points, which then forces $a=b$, which again leads to the regular octahedron.

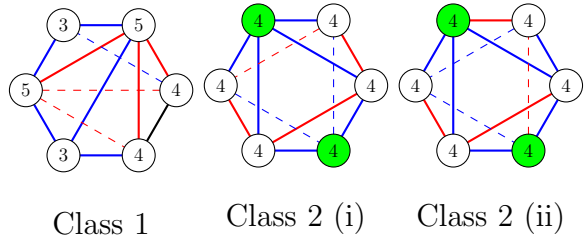
\begin{center}
\begin{tikzpicture}[scale=1]
    
    \coordinate (u1) at (1,0);
    \coordinate (u2) at ({cos(360/6)},{sin(360/6)});
    \coordinate (u3) at ({cos(2*360/6)},{sin(2*360/6)});
    \coordinate (u4) at ({cos(3*360/6)},{sin(3*360/6)});
    \coordinate (u5) at ({cos(4*360/6)},{sin(4*360/6)});
    \coordinate (u6) at ({cos(5*360/6)},{sin(5*360/6)});
    
    \begin{scope}[thick]
    \draw[red] (u1)--(u2);
    \draw[blue] (u2)--(u3);
    \draw[blue] (u3)--(u4);
    \draw[blue] (u4)--(u5);
    \draw[blue] (u5)--(u6);
    \draw[black] (u6)--(u1);
    \end{scope}
    
        \begin{scope}[thick]
   \draw[red] (u2)--(u4);
    \draw[blue] (u2)--(u5);
    \draw[red] (u2)--(u6);
    \end{scope}
    
       \begin{scope}[dashed]
    \draw[blue] (u1)--(u3);
    \draw[red] (u1)--(u4);
    \draw[red] (u4)--(u6);
    \end{scope}
    
     \node[circle,draw,fill=white,scale=0.6] (d1) at (u1){4};
    \node[circle,draw,fill=white,scale=0.6] (d2) at (u2){5};
    \node[circle,draw,fill=white,scale=0.6] (d3) at (u3){3};
    \node[circle,draw,fill=white,scale=0.6] (d4) at (u4){5};
   \node[circle,draw,fill=white,scale=0.6] (d5) at (u5){3};
    \node[circle,draw,fill=white,scale=0.6] (d6) at (u6){4};

      \coordinate (L1) at (0,{-sqrt(3)/2});
        \node[yshift=-9mm] at (L1) {Class 1};

\end{tikzpicture}
\begin{tikzpicture}[scale=1]
    
    \coordinate (u1) at (1,0);
    \coordinate (u2) at ({cos(360/6)},{sin(360/6)});
    \coordinate (u3) at ({cos(2*360/6)},{sin(2*360/6)});
    \coordinate (u4) at ({cos(3*360/6)},{sin(3*360/6)});
    \coordinate (u5) at ({cos(4*360/6)},{sin(4*360/6)});
    \coordinate (u6) at ({cos(5*360/6)},{sin(5*360/6)});
    
    \begin{scope}[thick]
    \draw[red] (u1)--(u2);
    \draw[blue] (u2)--(u3);
    \draw[blue] (u3)--(u4);
    \draw[red] (u4)--(u5);
    \draw[blue] (u5)--(u6);
    \draw[blue] (u6)--(u1);
    \end{scope}
    
        \begin{scope}[thick]
   \draw[blue] (u1)--(u3);
    \draw[blue] (u3)--(u5);
    \draw[red] (u5)--(u1);
    \end{scope}
    
       \begin{scope}[dashed]
    \draw[red] (u2)--(u4);
    \draw[blue] (u4)--(u6);
    \draw[blue] (u6)--(u2);
    \end{scope}
    
     \node[circle,draw,fill=white,scale=0.6] (d1) at (u1){4};
    \node[circle,draw,fill=white,scale=0.6] (d2) at (u2){4};
    \node[circle,draw,fill=green,scale=0.6] (d3) at (u3){4};
    \node[circle,draw,fill=white,scale=0.6] (d4) at (u4){4};
   \node[circle,draw,fill=white,scale=0.6] (d5) at (u5){4};
    \node[circle,draw,fill=green,scale=0.6] (d6) at (u6){4};

      \coordinate (L1) at (0,{-sqrt(3)/2});
        \node[yshift=-9mm] at (L1) {Class 2 (i)};

\end{tikzpicture}
\begin{tikzpicture}[scale=1]
    
    \coordinate (u1) at (1,0);
    \coordinate (u2) at ({cos(360/6)},{sin(360/6)});
    \coordinate (u3) at ({cos(2*360/6)},{sin(2*360/6)});
    \coordinate (u4) at ({cos(3*360/6)},{sin(3*360/6)});
    \coordinate (u5) at ({cos(4*360/6)},{sin(4*360/6)});
    \coordinate (u6) at ({cos(5*360/6)},{sin(5*360/6)});
    
    \begin{scope}[thick]
    \draw[blue] (u1)--(u2);
    \draw[red] (u2)--(u3);
    \draw[blue] (u3)--(u4);
    \draw[red] (u4)--(u5);
    \draw[blue] (u5)--(u6);
    \draw[blue] (u6)--(u1);
    \end{scope}
    
        \begin{scope}[thick]
   \draw[blue] (u1)--(u3);
    \draw[blue] (u3)--(u5);
    \draw[red] (u5)--(u1);
    \end{scope}
    
       \begin{scope}[dashed]
    \draw[blue] (u2)--(u4);
    \draw[blue] (u4)--(u6);
    \draw[red] (u6)--(u2);
    \end{scope}
    
     \node[circle,draw,fill=white,scale=0.6] (d1) at (u1){4};
    \node[circle,draw,fill=white,scale=0.6] (d2) at (u2){4};
    \node[circle,draw,fill=green,scale=0.6] (d3) at (u3){4};
    \node[circle,draw,fill=white,scale=0.6] (d4) at (u4){4};
   \node[circle,draw,fill=white,scale=0.6] (d5) at (u5){4};
    \node[circle,draw,fill=green,scale=0.6] (d6) at (u6){4};

      \coordinate (L1) at (0,{-sqrt(3)/2});
        \node[yshift=-9mm] at (L1) {Class 2 (ii)};    
\end{tikzpicture}
\end{center}

\section{Proof of Theorem \ref{7vertices}}

A convex polytope in $\R^3$ is \emph{simplicial} if it has triangular facets. There are precisely 5 nonisomorphic combinatorial types of simplicial convex  polytopes with 7 vertices (see, for example, the article \cite{BrittonDunitz1973} by Britton and Dunitz). One of these classes is that of the pentagonal bipyramid. Note that each polytope in any of the other four classes has a degree 3 vertex. For the reader's convenience, we include a diagram of the four simplicial classes of polytopes that have 7 vertices and at least one  degree 3 vertex. (These figures are recreations of the corresponding ones in \cite{BrittonDunitz1973}.) We then describe the conclusions in each class. 

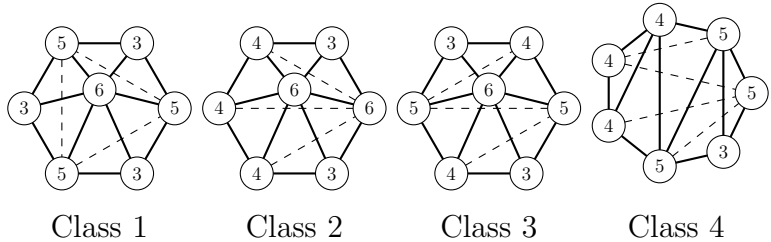
\begin{center}
\begin{tikzpicture}[scale=1]
    
    \coordinate (v1) at (1,0);
    \coordinate (v2) at (1/2,{sqrt(3)/2});
    \coordinate (v3) at (-1/2,{sqrt(3)/2});
    \coordinate (v4) at (-1,0);
    \coordinate (v5) at (-1/2,{-sqrt(3)/2});
    \coordinate (v6) at (1/2,{-sqrt(3)/2});
    \coordinate (v7) at (0,0.25);
        
    \begin{scope}[thick]
    \draw (v1)--(v2);
    \draw (v2)--(v3);
    \draw (v3)--(v4);
    \draw (v4)--(v5);
    \draw (v5)--(v6);
    \draw (v6)--(v1);
    \draw (v7)--(v1);
    \draw (v7)--(v2);
    \draw (v7)--(v3);
    \draw (v7)--(v4);
    \draw (v7)--(v5);
    \draw (v7)--(v6);
    \end{scope}
    
       \begin{scope}[dashed]
    \draw (v1)--(v3);
    \draw (v3)--(v5);
    \draw (v5)--(v1);
    \end{scope}
    
     \node[circle,draw,fill=white,scale=0.6] (c1) at (v1){5};
    \node[circle,draw,fill=white,scale=0.6] (c2) at (v2){3};
    \node[circle,draw,fill=white,scale=0.6] (c3) at (v3){5};
    \node[circle,draw,fill=white,scale=0.6] (c4) at (v4){3};
    \node[circle,draw,fill=white,scale=0.6] (c5) at (v5){5};
    \node[circle,draw,fill=white,scale=0.6] (c6) at (v6){3};
    \node[circle,draw,fill=white,scale=0.6] (c7) at (v7){6};
    
    \coordinate (L1) at (0,{-sqrt(3)/2});
    \node[yshift=-7mm] at (L1) {Class 1};    
    
\end{tikzpicture} 
\begin{tikzpicture}[scale=1]
    
    \begin{scope}[thick]
    \draw (v1)--(v2);
    \draw (v2)--(v3);
    \draw (v3)--(v4);
    \draw (v4)--(v5);
    \draw (v5)--(v6);
    \draw (v6)--(v1);
    \draw (v7)--(v1);
    \draw (v7)--(v2);
    \draw (v7)--(v3);
    \draw (v7)--(v4);
    \draw (v7)--(v5);
    \draw (v7)--(v6);
    \end{scope}
    
       \begin{scope}[dashed]
    \draw (v1)--(v3);
    \draw (v4)--(v1);
    \draw (v5)--(v1);
    \end{scope}
    
     \node[circle,draw,fill=white,scale=0.6] (c1) at (v1){6};
    \node[circle,draw,fill=white,scale=0.6] (c2) at (v2){3};
    \node[circle,draw,fill=white,scale=0.6] (c3) at (v3){4};
    \node[circle,draw,fill=white,scale=0.6] (c4) at (v4){4};
    \node[circle,draw,fill=white,scale=0.6] (c5) at (v5){4};
    \node[circle,draw,fill=white,scale=0.6] (c6) at (v6){3};
    \node[circle,draw,fill=white,scale=0.6] (c7) at (v7){6};
    
        \node[yshift=-7mm] at (L1) {Class 2};    

\end{tikzpicture}
\begin{tikzpicture}[scale=1]
    
    \begin{scope}[thick]
    \draw (v1)--(v2);
    \draw (v2)--(v3);
    \draw (v3)--(v4);
    \draw (v4)--(v5);
    \draw (v5)--(v6);
    \draw (v6)--(v1);
    \draw (v7)--(v1);
    \draw (v7)--(v2);
    \draw (v7)--(v3);
    \draw (v7)--(v4);
    \draw (v7)--(v5);
    \draw (v7)--(v6);
    \end{scope}
    
       \begin{scope}[dashed]
    \draw (v4)--(v2);
    \draw (v4)--(v1);
    \draw (v5)--(v1);
    \end{scope}
    
     \node[circle,draw,fill=white,scale=0.6] (c1) at (v1){5};
    \node[circle,draw,fill=white,scale=0.6] (c2) at (v2){4};
    \node[circle,draw,fill=white,scale=0.6] (c3) at (v3){3};
    \node[circle,draw,fill=white,scale=0.6] (c4) at (v4){5};
    \node[circle,draw,fill=white,scale=0.6] (c5) at (v5){4};
    \node[circle,draw,fill=white,scale=0.6] (c6) at (v6){3};
    \node[circle,draw,fill=white,scale=0.6] (c7) at (v7){6};

    \node[yshift=-7mm] at (L1) {Class 3};    
    
\end{tikzpicture}
\begin{tikzpicture}[scale=1]
    
    \coordinate (u1) at (1,0);
    \coordinate (u2) at ({cos(360/7)},{sin(360/7)});
    \coordinate (u3) at ({cos(2*360/7)},{sin(2*360/7)});
    \coordinate (u4) at ({cos(3*360/7)},{sin(3*360/7)});
    \coordinate (u5) at ({cos(4*360/7)},{sin(4*360/7)});
    \coordinate (u6) at ({cos(5*360/7)},{sin(5*360/7)});
    \coordinate (u7) at ({cos(6*360/7)},{sin(6*360/7)});
    
    \begin{scope}[thick]
    \draw (u1)--(u2);
    \draw (u2)--(u3);
    \draw (u3)--(u4);
    \draw (u4)--(u5);
    \draw (u5)--(u6);
    \draw (u6)--(u7);
    \draw (u7)--(u1);
    \end{scope}
    
        \begin{scope}[thick]
    \draw (u2)--(u7);
    \draw (u2)--(u6);
    \draw (u3)--(u6);
    \draw (u3)--(u5);
    \end{scope}
    
       \begin{scope}[dashed]
    \draw (u2)--(u4);
    \draw (u4)--(u1);
    \draw (u5)--(u1);
    \draw (u6)--(u1);
    \end{scope}
    
     \node[circle,draw,fill=white,scale=0.6] (d1) at (u1){5};
    \node[circle,draw,fill=white,scale=0.6] (d2) at (u2){5};
    \node[circle,draw,fill=white,scale=0.6] (d3) at (u3){4};
    \node[circle,draw,fill=white,scale=0.6] (d4) at (u4){4};
    \node[circle,draw,fill=white,scale=0.6] (d5) at (u5){4};
    \node[circle,draw,fill=white,scale=0.6] (d6) at (u6){5};
    \node[circle,draw,fill=white,scale=0.6] (d7) at (u7){3};
    
        \node[yshift=-9mm] at (L1) {Class 4};    
\end{tikzpicture}
\end{center}

In each of these 4 classes, every plausible 2-coloring is ruled out by Defect A, save for a single coloring in Class 2, shown below.
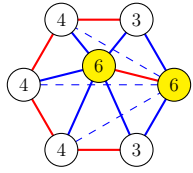
\begin{center}
\begin{tikzpicture}[scale=1]
    \begin{scope}[thick]
    \draw[blue] (v1)--(v2);
    \draw[red] (v2)--(v3);
    \draw[red] (v3)--(v4);
    \draw[red] (v4)--(v5);
    \draw[red] (v5)--(v6);
    \draw[blue] (v6)--(v1);
    \draw[red] (v7)--(v1);
    \draw[blue] (v7)--(v2);
    \draw[blue] (v7)--(v3);
    \draw[blue] (v7)--(v4);
    \draw[blue] (v7)--(v5);
    \draw[blue] (v7)--(v6);
    \end{scope}
    
       \begin{scope}[dashed]
    \draw[blue] (v1)--(v3);
    \draw[blue] (v4)--(v1);
    \draw[blue] (v5)--(v1);
    \end{scope}
    
     \node[circle,draw,fill=yellow,scale=0.6] (c1) at (v1){6};
    \node[circle,draw,fill=white,scale=0.6] (c2) at (v2){3};
    \node[circle,draw,fill=white,scale=0.6] (c3) at (v3){4};
    \node[circle,draw,fill=white,scale=0.6] (c4) at (v4){4};
    \node[circle,draw,fill=white,scale=0.6] (c5) at (v5){4};
    \node[circle,draw,fill=white,scale=0.6] (c6) at (v6){3};
    \node[circle,draw,fill=yellow,scale=0.6] (c7) at (v7){6};
    
\end{tikzpicture}
\end{center}
Notice that the two degree 6 vertices are antipodal by Lemma \ref{rule-out-lemma}. This  makes the chain of vertices connected with red edges an impossible configuration. 

We have thus shown that the surface area maximizer in $\mathcal{M}_7$ must be combinatorially equivalent to the pentagonal bipyramid.  By \cite[Cor. 2]{DHL}, the maximum surface area bipyramid has two vertices at the north and south poles and five more forming an equilateral pentagon in the equator. It has congruent isosceles facets and  surface area $\frac{5}{4}\sqrt{50-6\sqrt{5}}\approx 7.56$. This concludes the proof of Theorem \ref{7vertices}. \qed



\section{Proof of Theorem \ref{mainThm}}

\subsection{Simplicial polytopes with eight  vertices}

Britton and Dunitz   \cite{BrittonDunitz1973} provided graphs of all 257 combinatorial types of polytopes with eight vertices. Fourteen of these classes are simplicial. We recreate their graphs in the figure below.

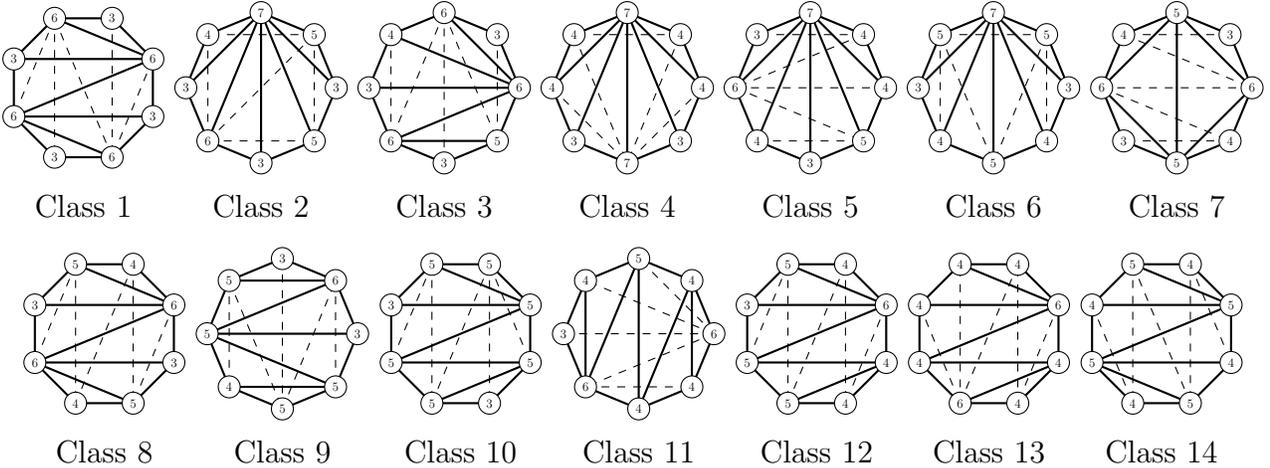
\begin{center}
\begin{tikzpicture}[scale=1]
    \coordinate (v1) at (0.9239795325,0.3826834324);
    \coordinate (v2) at (0.3826834324,0.9239795325);
    \coordinate (v3) at (-0.3826834324,0.9239795325);
    \coordinate (v4) at (-0.9239795325,0.3826834324);
    \coordinate (v5) at (-0.9239795325,-0.3826834324);
    \coordinate (v6) at (-0.3826834324,-0.9239795325);
    \coordinate (v7) at (0.3826834324,-0.9239795325);
    \coordinate (v8) at (0.9239795325,-0.3826834324); 

       \begin{scope}[dashed]
    \draw (v1)--(v7);
    \draw (v2)--(v7);
    \draw (v3)--(v7);
    \draw (v3)--(v6);
    \draw (v3)--(v5);
    \end{scope}
    
    \begin{scope}[thick]
    \draw (v1)--(v2);
    \draw (v2)--(v3);
    \draw (v3)--(v4);
    \draw (v4)--(v5);
    \draw (v5)--(v6);
    \draw (v6)--(v7);
    \draw (v7)--(v8);
    \draw (v8)--(v1);
    \draw (v1)--(v3);
    \draw (v1)--(v4);
    \draw (v1)--(v5);
    \draw (v5)--(v8);
    \draw (v5)--(v7);
    \end{scope}
    
    \node[circle,draw,fill=white,scale=0.4] (c1) at (v1){6};
    \node[circle,draw,fill=white,scale=0.4] (c2) at (v2){3};
    \node[circle,draw,fill=white,scale=0.4] (c3) at (v3){6};
    \node[circle,draw,fill=white,scale=0.4] (c4) at (v4){3};
    \node[circle,draw,fill=white,scale=0.4] (c5) at (v5){6};
    \node[circle,draw,fill=white,scale=0.4] (c6) at (v6){3};
    \node[circle,draw,fill=white,scale=0.4] (c7) at (v7){6};
    \node[circle,draw,fill=white,scale=0.4] (c7) at (v8){3};
    
    \coordinate (L1) at (0,{-sqrt(3)/2});
    \node[yshift=-7mm] at (L1) {Class 1};   
\end{tikzpicture}    
\begin{tikzpicture}[scale=1]
    
    \coordinate (v1) at (1,0);
    \coordinate (v2) at ({sqrt(2)/2},{sqrt(2)/2});
    \coordinate (v3) at (0,1);
    \coordinate (v4) at ({-sqrt(2)/2},{sqrt(2)/2});
    \coordinate (v5) at (-1,0);
    \coordinate (v6) at ({-sqrt(2)/2},{-sqrt(2)/2});
    \coordinate (v7) at (0,-1);
    \coordinate (v8) at ({sqrt(2)/2},{-sqrt(2)/2});

   \begin{scope}[dashed]
    \draw (v2)--(v4);
    \draw (v2)--(v6);
    \draw (v2)--(v8);
    \draw (v4)--(v6);
    \draw (v6)--(v8);
    \end{scope}
        
    \begin{scope}[thick]
    \draw (v1)--(v2);
    \draw (v1)--(v3);
    \draw (v2)--(v3);
    \draw (v3)--(v4);
    \draw (v4)--(v5);
    \draw (v5)--(v6);
    \draw (v6)--(v7);
    \draw (v7)--(v8);
    \draw (v1)--(v8);
    \draw (v3)--(v5);
    \draw (v3)--(v6);
    \draw (v3)--(v7);
    \draw (v3)--(v8);
    \end{scope}
    
    \node[circle,draw,fill=white,scale=0.4] (c1) at (v1){3};
    \node[circle,draw,fill=white,scale=0.4] (c2) at (v2){5};
    \node[circle,draw,fill=white,scale=0.4] (c3) at (v3){7};
    \node[circle,draw,fill=white,scale=0.4] (c4) at (v4){4};
    \node[circle,draw,fill=white,scale=0.4] (c5) at (v5){3};
    \node[circle,draw,fill=white,scale=0.4] (c6) at (v6){6};
    \node[circle,draw,fill=white,scale=0.4] (c7) at (v7){3};
    \node[circle,draw,fill=white,scale=0.4] (c7) at (v8){5};
    
    \coordinate (L1) at (0,{-sqrt(3)/2});
    \node[yshift=-7mm] at (L1) {Class 2};    
    
\end{tikzpicture}
\begin{tikzpicture}[scale=1]

    \coordinate (v1) at (1,0);
    \coordinate (v2) at ({sqrt(2)/2},{sqrt(2)/2});
    \coordinate (v3) at (0,1);
    \coordinate (v4) at ({-sqrt(2)/2},{sqrt(2)/2});
    \coordinate (v5) at (-1,0);
    \coordinate (v6) at ({-sqrt(2)/2},{-sqrt(2)/2});
    \coordinate (v7) at (0,-1);
    \coordinate (v8) at ({sqrt(2)/2},{-sqrt(2)/2});

    \begin{scope}[dashed]
    \draw (v1)--(v8);
    \draw (v3)--(v8);
    \draw (v3)--(v7);
    \draw (v3)--(v6);
    \draw (v4)--(v6);
    \draw (v2)--(v8);
    \end{scope}
    
    \begin{scope}[thick]
    \draw (v1)--(v2);
    \draw (v2)--(v3);
    \draw (v3)--(v4);
    \draw (v4)--(v5);
    \draw (v5)--(v6);
    \draw (v6)--(v7);
    \draw (v7)--(v8);
    \draw (v8)--(v1);
    \draw (v1)--(v3);
    \draw (v1)--(v4);
    \draw (v1)--(v5);
    \draw (v1)--(v6);
    \draw (v6)--(v8);
    \end{scope}
    
    \node[circle,draw,fill=white,scale=0.4] (c1) at (v1){6};
    \node[circle,draw,fill=white,scale=0.4] (c2) at (v2){3};
    \node[circle,draw,fill=white,scale=0.4] (c3) at (v3){6};
    \node[circle,draw,fill=white,scale=0.4] (c4) at (v4){4};
    \node[circle,draw,fill=white,scale=0.4] (c5) at (v5){3};
    \node[circle,draw,fill=white,scale=0.4] (c6) at (v6){6};
    \node[circle,draw,fill=white,scale=0.4] (c7) at (v7){3};
    \node[circle,draw,fill=white,scale=0.4] (c7) at (v8){5};
    
    \coordinate (L1) at (0,{-sqrt(3)/2});
    \node[yshift=-7mm] at (L1) {Class 3};   
    
\end{tikzpicture}    
\begin{tikzpicture}[scale=1]
    
    \coordinate (v1) at (1,0);
    \coordinate (v2) at ({sqrt(2)/2},{sqrt(2)/2});
    \coordinate (v3) at (0,1);
    \coordinate (v4) at ({-sqrt(2)/2},{sqrt(2)/2});
    \coordinate (v5) at (-1,0);
    \coordinate (v6) at ({-sqrt(2)/2},{-sqrt(2)/2});
    \coordinate (v7) at (0,-1);
    \coordinate (v8) at ({sqrt(2)/2},{-sqrt(2)/2});

  \begin{scope}[dashed]
    \draw (v2)--(v4);
    \draw (v1)--(v7);
    \draw (v2)--(v7);
    \draw (v4)--(v7);
    \draw (v5)--(v7);
    \end{scope}
        
    \begin{scope}[thick]
    \draw (v1)--(v2);
    \draw (v1)--(v3);
    \draw (v1)--(v8);
    \draw (v2)--(v3);
    \draw (v3)--(v4);
    \draw (v4)--(v5);
    \draw (v5)--(v6);
    \draw (v6)--(v7);
    \draw (v7)--(v8);
    \draw (v3)--(v5);
    \draw (v3)--(v6);
    \draw (v3)--(v7);
    \draw (v3)--(v8);
    \end{scope}
    
    \node[circle,draw,fill=white,scale=0.4] (c1) at (v1){4};
    \node[circle,draw,fill=white,scale=0.4] (c2) at (v2){4};
    \node[circle,draw,fill=white,scale=0.4] (c3) at (v3){7};
    \node[circle,draw,fill=white,scale=0.4] (c4) at (v4){4};
    \node[circle,draw,fill=white,scale=0.4] (c5) at (v5){4};
    \node[circle,draw,fill=white,scale=0.4] (c6) at (v6){3};
    \node[circle,draw,fill=white,scale=0.4] (c7) at (v7){7};
    \node[circle,draw,fill=white,scale=0.4] (c7) at (v8){3};
    
    \coordinate (L1) at (0,{-sqrt(3)/2});
    \node[yshift=-7mm] at (L1) {Class 4};    
    
\end{tikzpicture}
\begin{tikzpicture}[scale=1]
    
    \coordinate (v1) at (1,0);
    \coordinate (v2) at ({sqrt(2)/2},{sqrt(2)/2});
    \coordinate (v3) at (0,1);
    \coordinate (v4) at ({-sqrt(2)/2},{sqrt(2)/2});
    \coordinate (v5) at (-1,0);
    \coordinate (v6) at ({-sqrt(2)/2},{-sqrt(2)/2});
    \coordinate (v7) at (0,-1);
    \coordinate (v8) at ({sqrt(2)/2},{-sqrt(2)/2});

    \begin{scope}[dashed]
    \draw (v2)--(v4);
    \draw (v2)--(v5);
    \draw (v5)--(v1);
    \draw (v5)--(v8);
    \draw (v6)--(v8);
    \end{scope}
    
        \begin{scope}[thick]
    \draw (v1)--(v2);
    \draw (v2)--(v3);
    \draw (v3)--(v4);
    \draw (v4)--(v5);
    \draw (v5)--(v6);
    \draw (v6)--(v7);
    \draw (v7)--(v8);
    \draw (v8)--(v1);
    \draw (v3)--(v5);
    \draw (v3)--(v6);
    \draw (v3)--(v7);
    \draw (v3)--(v8);
    \draw (v3)--(v1);
    \end{scope}
     
    \node[circle,draw,fill=white,scale=0.4] (c1) at (v1){4};
    \node[circle,draw,fill=white,scale=0.4] (c2) at (v2){4};
    \node[circle,draw,fill=white,scale=0.4] (c3) at (v3){7};
    \node[circle,draw,fill=white,scale=0.4] (c4) at (v4){3};
    \node[circle,draw,fill=white,scale=0.4] (c5) at (v5){6};
    \node[circle,draw,fill=white,scale=0.4] (c6) at (v6){4};
    \node[circle,draw,fill=white,scale=0.4] (c7) at (v7){3};
    \node[circle,draw,fill=white,scale=0.4] (c7) at (v8){5};
    
    \coordinate (L1) at (0,{-sqrt(3)/2});
    \node[yshift=-7mm] at (L1) {Class 5}; 
\end{tikzpicture}
\begin{tikzpicture}[scale=1]
    
    \coordinate (v1) at (1,0);
    \coordinate (v2) at ({sqrt(2)/2},{sqrt(2)/2});
    \coordinate (v3) at (0,1);
    \coordinate (v4) at ({-sqrt(2)/2},{sqrt(2)/2});
    \coordinate (v5) at (-1,0);
    \coordinate (v6) at ({-sqrt(2)/2},{-sqrt(2)/2});
    \coordinate (v7) at (0,-1);
    \coordinate (v8) at ({sqrt(2)/2},{-sqrt(2)/2});

    \begin{scope}[dashed]
    \draw (v2)--(v4);
    \draw (v2)--(v7);
    \draw (v2)--(v8);
    \draw (v4)--(v6);
    \draw (v4)--(v7);
    \end{scope}
        
    \begin{scope}[thick]
    \draw (v1)--(v2);
    \draw (v2)--(v3);
    \draw (v3)--(v4);
    \draw (v4)--(v5);
    \draw (v5)--(v6);
    \draw (v6)--(v7);
    \draw (v7)--(v8);
    \draw (v8)--(v1);
    \draw (v3)--(v1);
    \draw (v3)--(v5);
    \draw (v3)--(v6);
    \draw (v3)--(v7);
    \draw (v3)--(v8);
    \end{scope}
    
    \node[circle,draw,fill=white,scale=0.4] (c1) at (v1){3};
    \node[circle,draw,fill=white,scale=0.4] (c2) at (v2){5};
    \node[circle,draw,fill=white,scale=0.4] (c3) at (v3){7};
    \node[circle,draw,fill=white,scale=0.4] (c4) at (v4){5};
    \node[circle,draw,fill=white,scale=0.4] (c5) at (v5){3};
    \node[circle,draw,fill=white,scale=0.4] (c6) at (v6){4};
    \node[circle,draw,fill=white,scale=0.4] (c7) at (v7){5};
    \node[circle,draw,fill=white,scale=0.4] (c7) at (v8){4};
    
    \coordinate (L1) at (0,{-sqrt(3)/2});
    \node[yshift=-7mm] at (L1) {Class 6};    
    
\end{tikzpicture}
\begin{tikzpicture}[scale=1]
    
    \coordinate (v1) at (1,0);
    \coordinate (v2) at ({sqrt(2)/2},{sqrt(2)/2});
    \coordinate (v3) at (0,1);
    \coordinate (v4) at ({-sqrt(2)/2},{sqrt(2)/2});
    \coordinate (v5) at (-1,0);
    \coordinate (v6) at ({-sqrt(2)/2},{-sqrt(2)/2});
    \coordinate (v7) at (0,-1);
    \coordinate (v8) at ({sqrt(2)/2},{-sqrt(2)/2});

 \begin{scope}[dashed]
    \draw (v2)--(v4);
    \draw (v4)--(v1);
    \draw (v1)--(v5);
    \draw (v5)--(v8);
    \draw (v8)--(v6);
    \end{scope}
    
    \begin{scope}[thick]
    \draw (v1)--(v2);
    \draw (v2)--(v3);
    \draw (v3)--(v4);
    \draw (v4)--(v5);
    \draw (v5)--(v6);
    \draw (v6)--(v7);
    \draw (v7)--(v8);
    \draw (v8)--(v1);
    \draw (v1)--(v3);
    \draw (v3)--(v5);
    \draw (v5)--(v7);
    \draw (v7)--(v1);
    \draw (v3)--(v7);
    \end{scope}
    
    \node[circle,draw,fill=white,scale=0.4] (c1) at (v1){6};
    \node[circle,draw,fill=white,scale=0.4] (c2) at (v2){3};
    \node[circle,draw,fill=white,scale=0.4] (c3) at (v3){5};
    \node[circle,draw,fill=white,scale=0.4] (c4) at (v4){4};
    \node[circle,draw,fill=white,scale=0.4] (c5) at (v5){6};
    \node[circle,draw,fill=white,scale=0.4] (c6) at (v6){3};
    \node[circle,draw,fill=white,scale=0.4] (c7) at (v7){5};
    \node[circle,draw,fill=white,scale=0.4] (c7) at (v8){4};
    
    \coordinate (L1) at (0,{-sqrt(3)/2});
    \node[yshift=-7mm] at (L1) {Class 7};    
    
\end{tikzpicture}
\end{center}

\begin{center}
\begin{tikzpicture}[scale=1]

    \coordinate (v1) at (0.9239795325,0.3826834324);
    \coordinate (v2) at (0.3826834324,0.9239795325);
    \coordinate (v3) at (-0.3826834324,0.9239795325);
    \coordinate (v4) at (-0.9239795325,0.3826834324);
    \coordinate (v5) at (-0.9239795325,-0.3826834324);
    \coordinate (v6) at (-0.3826834324,-0.9239795325);
    \coordinate (v7) at (0.3826834324,-0.9239795325);
    \coordinate (v8) at (0.9239795325,-0.3826834324);  

    \begin{scope}[dashed]
    \draw (v1)--(v7);
    \draw (v2)--(v6);
    \draw (v2)--(v7);
    \draw (v3)--(v5);
    \draw (v3)--(v6);
    \end{scope}
    
    \begin{scope}[thick]
    \draw (v1)--(v2);
    \draw (v2)--(v3);
    \draw (v3)--(v4);
    \draw (v4)--(v5);
    \draw (v5)--(v6);
    \draw (v6)--(v7);
    \draw (v7)--(v8);
    \draw (v8)--(v1);
    \draw (v1)--(v3);
    \draw (v1)--(v4);
    \draw (v1)--(v5);
    \draw (v5)--(v7);
    \draw (v5)--(v8);
    \end{scope}
    
    \node[circle,draw,fill=white,scale=0.4] (c1) at (v1){6};
    \node[circle,draw,fill=white,scale=0.4] (c2) at (v2){4};
    \node[circle,draw,fill=white,scale=0.4] (c3) at (v3){5};
    \node[circle,draw,fill=white,scale=0.4] (c4) at (v4){3};
    \node[circle,draw,fill=white,scale=0.4] (c5) at (v5){6};
    \node[circle,draw,fill=white,scale=0.4] (c6) at (v6){4};
    \node[circle,draw,fill=white,scale=0.4] (c7) at (v7){5};
    \node[circle,draw,fill=white,scale=0.4] (c7) at (v8){3};
    
    \coordinate (L1) at (0,{-sqrt(3)/2});
    \node[yshift=-7mm] at (L1) {Class 8};    
    
\end{tikzpicture}
  \begin{tikzpicture}[scale=1]
      
    \coordinate (v1) at (1,0);
    \coordinate (v2) at ({sqrt(2)/2},{sqrt(2)/2});
    \coordinate (v3) at (0,1);
    \coordinate (v4) at ({-sqrt(2)/2},{sqrt(2)/2});
    \coordinate (v5) at (-1,0);
    \coordinate (v6) at ({-sqrt(2)/2},{-sqrt(2)/2});
    \coordinate (v7) at (0,-1);
    \coordinate (v8) at ({sqrt(2)/2},{-sqrt(2)/2});

    \begin{scope}[dashed]
      \draw (v2)--(v8);
      \draw (v4)--(v6);
      \draw (v3)--(v7);
      \draw (v4)--(v7);
      \draw (v7)--(v2);
      \end{scope}
      
      \begin{scope}[thick]
      \draw (v1)--(v2);
      \draw (v2)--(v3);
      \draw (v3)--(v4);
      \draw (v4)--(v5);
      \draw (v5)--(v6);
      \draw (v6)--(v7);
      \draw (v7)--(v8);
      \draw (v8)--(v1);
      \draw (v4)--(v2);
      \draw (v2)--(v5);
      \draw (v5)--(v1);
      \draw (v5)--(v8);
      \draw (v8)--(v6);
      \end{scope}
      
     \node[circle,draw,fill=white,scale=0.4] (c1) at (v1){3};
    \node[circle,draw,fill=white,scale=0.4] (c2) at (v2){6};
    \node[circle,draw,fill=white,scale=0.4] (c3) at (v3){3};
    \node[circle,draw,fill=white,scale=0.4] (c4) at (v4){5};
    \node[circle,draw,fill=white,scale=0.4] (c5) at (v5){5};
    \node[circle,draw,fill=white,scale=0.4] (c6) at (v6){4};
    \node[circle,draw,fill=white,scale=0.4] (c7) at (v7){5};
    \node[circle,draw,fill=white,scale=0.4] (c7) at (v8){5};
      
    \coordinate (L1) at (0,{-sqrt(3)/2});
    \node[yshift=-7mm] at (L1) {Class 9};
      
  \end{tikzpicture}
\begin{tikzpicture}[scale=1]
    
    \coordinate (v1) at (0.9239795325,0.3826834324);
    \coordinate (v2) at (0.3826834324,0.9239795325);
    \coordinate (v3) at (-0.3826834324,0.9239795325);
    \coordinate (v4) at (-0.9239795325,0.3826834324);
    \coordinate (v5) at (-0.9239795325,-0.3826834324);
    \coordinate (v6) at (-0.3826834324,-0.9239795325);
    \coordinate (v7) at (0.3826834324,-0.9239795325);
    \coordinate (v8) at (0.9239795325,-0.3826834324); 

      \begin{scope}[dashed]
    \draw (v2)--(v6);
    \draw (v2)--(v7);
    \draw (v2)--(v8);
    \draw (v3)--(v5);
    \draw (v3)--(v6);
    \end{scope}
    
    \begin{scope}[thick]
    \draw (v1)--(v2);
    \draw (v2)--(v3);
    \draw (v3)--(v4);
    \draw (v4)--(v5);
    \draw (v5)--(v6);
    \draw (v6)--(v7);
    \draw (v7)--(v8);
    \draw (v8)--(v1);
    \draw (v1)--(v3);
    \draw (v1)--(v4);
    \draw (v1)--(v5);
    \draw (v5)--(v8);
    \draw (v6)--(v8);
    \end{scope}
    
    \node[circle,draw,fill=white,scale=0.4] (c1) at (v1){5};
    \node[circle,draw,fill=white,scale=0.4] (c2) at (v2){5};
    \node[circle,draw,fill=white,scale=0.4] (c3) at (v3){5};
    \node[circle,draw,fill=white,scale=0.4] (c4) at (v4){3};
    \node[circle,draw,fill=white,scale=0.4] (c5) at (v5){5};
    \node[circle,draw,fill=white,scale=0.4] (c6) at (v6){5};
    \node[circle,draw,fill=white,scale=0.4] (c7) at (v7){3};
    \node[circle,draw,fill=white,scale=0.4] (c7) at (v8){5};
    
    \coordinate (L1) at (0,{-sqrt(3)/2});
    \node[yshift=-7mm] at (L1) {Class 10};    
    
\end{tikzpicture}
\begin{tikzpicture}[scale=1]
   
   \coordinate (v1) at (1,0);
    \coordinate (v2) at ({sqrt(2)/2},{sqrt(2)/2});
    \coordinate (v3) at (0,1);
    \coordinate (v4) at ({-sqrt(2)/2},{sqrt(2)/2});
    \coordinate (v5) at (-1,0);
    \coordinate (v6) at ({-sqrt(2)/2},{-sqrt(2)/2});
    \coordinate (v7) at (0,-1);
    \coordinate (v8) at ({sqrt(2)/2},{-sqrt(2)/2});
    
    \begin{scope}[thick]
    \draw (v1)--(v2);
    \draw (v2)--(v3);
    \draw (v3)--(v4);    
    \draw (v4)--(v5);    
    \draw (v5)--(v6);
    \draw (v6)--(v7);
    \draw (v7)--(v8);
    \draw (v8)--(v1);
    \draw (v4)--(v6);
    \draw (v6)--(v3);
    \draw (v3)--(v7);
    \draw (v7)--(v2);
    \draw (v2)--(v8);
    \end{scope}
   
   \begin{scope}[dashed]
   \draw (v1)--(v3);
   \draw (v1)--(v4);
   \draw (v1)--(v5);
   \draw (v1)--(v6);
   \draw (v6)--(v8);
   \end{scope}

    \node[circle,draw,fill=white,scale=0.4] (c1) at (v1){6};
    \node[circle,draw,fill=white,scale=0.4] (c2) at (v2){4};
    \node[circle,draw,fill=white,scale=0.4] (c3) at (v3){5};
    \node[circle,draw,fill=white,scale=0.4] (c4) at (v4){4};
    \node[circle,draw,fill=white,scale=0.4] (c5) at (v5){3};
    \node[circle,draw,fill=white,scale=0.4] (c6) at (v6){6};
    \node[circle,draw,fill=white,scale=0.4] (c7) at (v7){4};
    \node[circle,draw,fill=white,scale=0.4] (c7) at (v8){4};
    
    \coordinate (L1) at (0,{-sqrt(3)/2});
    \node[yshift=-7mm] at (L1) {Class 11};

\end{tikzpicture}
\begin{tikzpicture}[scale=1]

    \coordinate (v1) at (0.9239795325,0.3826834324);
    \coordinate (v2) at (0.3826834324,0.9239795325);
    \coordinate (v3) at (-0.3826834324,0.9239795325);
    \coordinate (v4) at (-0.9239795325,0.3826834324);
    \coordinate (v5) at (-0.9239795325,-0.3826834324);
    \coordinate (v6) at (-0.3826834324,-0.9239795325);
    \coordinate (v7) at (0.3826834324,-0.9239795325);
    \coordinate (v8) at (0.9239795325,-0.3826834324);
    
    \begin{scope}[thick]
    \draw (v1)--(v2);
    \draw (v2)--(v3);
    \draw (v3)--(v4);
    \draw (v4)--(v5);
    \draw (v5)--(v6);
    \draw (v6)--(v7);
    \draw (v7)--(v8);
    \draw (v8)--(v1);
    \draw (v1)--(v3);
    \draw (v1)--(v4);
    \draw (v1)--(v5);
    \draw (v5)--(v8);
    \draw (v6)--(v8);
    \end{scope}
    
    \begin{scope}[dashed]
    \draw (v1)--(v7);
    \draw (v2)--(v6);
    \draw (v2)--(v7);
    \draw (v3)--(v5);
    \draw (v3)--(v6);
    \end{scope}
    
    \node[circle,draw,fill=white,scale=0.4] (c1) at (v1){6};
    \node[circle,draw,fill=white,scale=0.4] (c2) at (v2){4};
    \node[circle,draw,fill=white,scale=0.4] (c3) at (v3){5};
    \node[circle,draw,fill=white,scale=0.4] (c4) at (v4){3};
    \node[circle,draw,fill=white,scale=0.4] (c5) at (v5){5};
    \node[circle,draw,fill=white,scale=0.4] (c6) at (v6){5};
    \node[circle,draw,fill=white,scale=0.4] (c7) at (v7){4};
    \node[circle,draw,fill=white,scale=0.4] (c7) at (v8){4};
    
    \coordinate (L1) at (0,{-sqrt(3)/2});
    \node[yshift=-7mm] at (L1) {Class 12};    
    
\end{tikzpicture}
\begin{tikzpicture}[scale=1]
  
    \coordinate (v1) at (0.9239795325,0.3826834324);
    \coordinate (v2) at (0.3826834324,0.9239795325);
    \coordinate (v3) at (-0.3826834324,0.9239795325);
    \coordinate (v4) at (-0.9239795325,0.3826834324);
    \coordinate (v5) at (-0.9239795325,-0.3826834324);
    \coordinate (v6) at (-0.3826834324,-0.9239795325);
    \coordinate (v7) at (0.3826834324,-0.9239795325);
    \coordinate (v8) at (0.9239795325,-0.3826834324); 
  
    \begin{scope}[thick]
    \draw (v1)--(v2);
    \draw (v2)--(v3);
    \draw (v3)--(v4);
    \draw (v4)--(v5);
    \draw (v5)--(v6);
    \draw (v6)--(v7);
    \draw (v7)--(v8);
    \draw (v8)--(v1);
    \draw (v1)--(v3);
    \draw (v1)--(v4);
    \draw (v1)--(v5);
    \draw (v5)--(v8);
    \draw (v8)--(v6);
    \end{scope}
    
    \begin{scope}[dashed]
    \draw (v1)--(v7);
    \draw (v7)--(v2);
    \draw (v2)--(v6);
    \draw (v6)--(v3);
    \draw (v4)--(v6);
    \end{scope}
    
    \node[circle,draw,fill=white,scale=0.4] (c1) at (v1){6};
    \node[circle,draw,fill=white,scale=0.4] (c2) at (v2){4};
    \node[circle,draw,fill=white,scale=0.4] (c3) at (v3){4};
    \node[circle,draw,fill=white,scale=0.4] (c4) at (v4){4};
    \node[circle,draw,fill=white,scale=0.4] (c5) at (v5){4};
    \node[circle,draw,fill=white,scale=0.4] (c6) at (v6){6};
    \node[circle,draw,fill=white,scale=0.4] (c7) at (v7){4};
    \node[circle,draw,fill=white,scale=0.4] (c7) at (v8){4};
    
    \coordinate (L1) at (0,{-sqrt(3)/2});
    \node[yshift=-7mm] at (L1) {Class 13};

\end{tikzpicture}
\begin{tikzpicture}[scale=1]
    
    \coordinate (v1) at (0.9239795325,0.3826834324);
    \coordinate (v2) at (0.3826834324,0.9239795325);
    \coordinate (v3) at (-0.3826834324,0.9239795325);
    \coordinate (v4) at (-0.9239795325,0.3826834324);
    \coordinate (v5) at (-0.9239795325,-0.3826834324);
    \coordinate (v6) at (-0.3826834324,-0.9239795325);
    \coordinate (v7) at (0.3826834324,-0.9239795325);
    \coordinate (v8) at (0.9239795325,-0.3826834324); 
        
    \begin{scope}[thick]
    \draw (v1)--(v2);
    \draw (v2)--(v3);
    \draw (v3)--(v4);
    \draw (v4)--(v5);
    \draw (v5)--(v6);
    \draw (v6)--(v7);
    \draw (v7)--(v8);
    \draw (v8)--(v1);
    \draw (v1)--(v3);
    \draw (v1)--(v4);
    \draw (v1)--(v5);
    \draw (v5)--(v7);
    \draw (v5)--(v8);
    \end{scope}
    
       \begin{scope}[dashed]
    \draw (v2)--(v7);
    \draw (v2)--(v8);
    \draw (v3)--(v6);
    \draw (v3)--(v7);
    \draw (v4)--(v6);
    \end{scope}
    
     \node[circle,draw,fill=white,scale=0.4] (c1) at (v1){5};
    \node[circle,draw,fill=white,scale=0.4] (c2) at (v2){4};
    \node[circle,draw,fill=white,scale=0.4] (c3) at (v3){5};
    \node[circle,draw,fill=white,scale=0.4] (c4) at (v4){4};
    \node[circle,draw,fill=white,scale=0.4] (c5) at (v5){5};
    \node[circle,draw,fill=white,scale=0.4] (c6) at (v6){4};
    \node[circle,draw,fill=white,scale=0.4] (c7) at (v7){5};
    \node[circle,draw,fill=white,scale=0.4] (c7) at (v8){4};
    
    \coordinate (L1) at (0,{-sqrt(3)/2});
    \node[yshift=-7mm] at (L1) {Class 14};    
    
\end{tikzpicture}
    \end{center}

We will show that most of the 14 simplicial classes do not admit an inscribable  realization which is equifacetal; for those that do, we compute  their surface areas directly, and at the end of the proof we select the largest value. We begin with the cases in which each facet is isosceles, and at the end we consider the potential cases  where each facet is equilateral.  Please note that each of the possible 2-colorings from Classes 2 through 7, as well as those of Classes 9 and 11, are ruled out by Defect A or Defect B.

\subsection{The cases of isosceles facets}

\subsubsection{Class 1}

The \emph{triakis tetrahedron} is a member of Class 1. It is constructed by gluing a tetrahedron to each facet of a tetrahedron. It is not \emph{inscribable}, meaning there is no combinatorially equivalent polytope whose vertices all lie in the sphere (see, for example, \cite{Padrol2016}).  Hence the surface area  maximizer cannot lie in Class 1. 



\subsubsection{Class 8}
Lemma \ref{rule-out-lemma} provides two sets of antipodal vertices, shaded green and yellow in the figure below. 
\begin{center}
\begin{tikzpicture}[scale=1]

    \coordinate (v1) at (0.9239795325,0.3826834324);
    \coordinate (v2) at (0.3826834324,0.9239795325);
    \coordinate (v3) at (-0.3826834324,0.9239795325);
    \coordinate (v4) at (-0.9239795325,0.3826834324);
    \coordinate (v5) at (-0.9239795325,-0.3826834324);
    \coordinate (v6) at (-0.3826834324,-0.9239795325);
    \coordinate (v7) at (0.3826834324,-0.9239795325);
    \coordinate (v8) at (0.9239795325,-0.3826834324);  

    \begin{scope}[dashed,thick]
    \draw[red] (v1)--(v7);
    \draw[red] (v2)--(v6);
    \draw[blue] (v2)--(v7);
    \draw[red] (v3)--(v5);
    \draw[blue] (v3)--(v6);
    \end{scope}
    
    \begin{scope}[thick]
    \draw[blue] (v1)--(v2);
    \draw[blue] (v2)--(v3);
    \draw[blue] (v3)--(v4);
    \draw[blue] (v4)--(v5);
    \draw[blue] (v5)--(v6);
    \draw[blue] (v6)--(v7);
    \draw[blue] (v7)--(v8);
    \draw[blue] (v8)--(v1);
    \draw[red] (v1)--(v3);
    \draw[blue] (v1)--(v4);
    \draw[red] (v1)--(v5);
    \draw[red] (v5)--(v7);
    \draw[blue] (v5)--(v8);
    \end{scope}
    
    \node[circle,draw,fill=green,scale=0.4] (c1) at (v1){$E$};
    \node[circle,draw,fill=yellow,scale=0.4] (c2) at (v2){$D$};
    \node[circle,draw,fill=white,scale=0.4] (c3) at (v3){$C$};
    \node[circle,draw,fill=white,scale=0.4] (c4) at (v4){$B$};
    \node[circle,draw,fill=yellow,scale=0.4] (c5) at (v5){$A$};
    \node[circle,draw,fill=green,scale=0.4] (c6) at (v6){$H$};
    \node[circle,draw,fill=white,scale=0.4] (c7) at (v7){$G$};
    \node[circle,draw,fill=white,scale=0.4] (c7) at (v8){$F$};
    
    \coordinate (L1) at (0,{-sqrt(3)/2});
    \node[yshift=-7mm] at (L1) {Class 8};    
\end{tikzpicture}
\end{center}

The vertices $A,D,E$ and $H$ lie in the same great circle and form a rectangle, with $|\overline{AE}|=a$ and $|\overline{AH}|=b$.  The triangles $\triangle(ABE), \triangle(AEF), \triangle(CDH)$ and $\triangle(DGH)$ are congruent isosceles triangles with $|\overline{BC}|=|\overline{FG}|=b$. Thus it must be that the vertices $A,B,E$ and $F$ coplanar, and the vertices $C,D,G$ and $H$ must be coplanar as well, which is impossible.


\subsubsection{Class 10}
None of the coloring defects are present in the coloring below. 

\begin{center}
    \begin{tikzpicture}[scale=1]
    
    \coordinate (v1) at (0.9239795325,0.3826834324);
    \coordinate (v2) at (0.3826834324,0.9239795325);
    \coordinate (v3) at (-0.3826834324,0.9239795325);
    \coordinate (v4) at (-0.9239795325,0.3826834324);
    \coordinate (v5) at (-0.9239795325,-0.3826834324);
    \coordinate (v6) at (-0.3826834324,-0.9239795325);
    \coordinate (v7) at (0.3826834324,-0.9239795325);
    \coordinate (v8) at (0.9239795325,-0.3826834324); 

      \begin{scope}[dashed,thick]
    \draw[red] (v2)--(v6);
    \draw[blue] (v2)--(v7);
    \draw[red] (v2)--(v8);
    \draw[red] (v3)--(v5);
    \draw[blue] (v3)--(v6);
    \end{scope}
    
    \begin{scope}[thick]
    \draw[blue] (v1)--(v2);
    \draw[blue] (v2)--(v3);
    \draw[blue] (v3)--(v4);
    \draw[blue] (v4)--(v5);
    \draw[blue] (v5)--(v6);
    \draw[blue] (v6)--(v7);
    \draw[blue] (v7)--(v8);
    \draw[blue] (v8)--(v1);
    \draw[red] (v1)--(v3);
    \draw[blue] (v1)--(v4);
    \draw[red] (v1)--(v5);
    \draw[blue] (v5)--(v8);
    \draw[red] (v6)--(v8);
    \end{scope}
    
    \node[circle,draw,fill=white,scale=0.4] (c1) at (v1){$E$};
    \node[circle,draw,fill=white,scale=0.4] (c2) at (v2){$D$};
    \node[circle,draw,fill=white,scale=0.4] (c3) at (v3){$C$};
    \node[circle,draw,fill=white,scale=0.4] (c4) at (v4){$B$};
    \node[circle,draw,fill=white,scale=0.4] (c5) at (v5){$A$};
    \node[circle,draw,fill=white,scale=0.4] (c6) at (v6){$H$};
    \node[circle,draw,fill=white,scale=0.4] (c7) at (v7){$G$};
    \node[circle,draw,fill=white,scale=0.4] (c7) at (v8){$F$};
    
    \coordinate (L1) at (0,{-sqrt(3)/2});
    \node[yshift=-7mm] at (L1) {Class 10};    
\end{tikzpicture}
\end{center}

Let $b=|\overline{AB}|$ and $a=|\overline{AE}|$.  Without loss of generality, we may assume that $B=e_3$ and that the plane containing the equilateral triangle $\triangle(ACE)$ is given by $z=h$, where $h\in(-1,1)$.  This allows us to calculate $a=\sqrt{3(1-h^2)}$ and $b=\sqrt{2(1-h)}$.  Since the triangles $\triangle(ACH), \triangle(CDE)$ and $\triangle(AFE)$ are isosceles, the plane containing $\triangle(DFH)$ is parallel to the plane containing $\triangle(ACE)$.  Since $\triangle(ACE)$ and $\triangle(DFH)$ are congruent, the plane containing $\triangle(DFH)$ is given by $z=-h$, and $G=-e_3$.  Furthermore, the congruency of the triangles in between forces the two equilateral triangles to be separated by distance $\pi/3$.  This allows us to calculate $b=|\overline{DE}|=\sqrt{1+3h^2}$. Setting this equal to the previous expression yields $2(1-h)=1+3h^2$, which has one feasible solution $h=1/3$.  Now we can calculate the surface area of this configuration:
\[
6a\sqrt{b^2-a^2/4}=8.
\]


\subsubsection{Class 12}  Lemma \ref{rule-out-lemma} provides a pair of antipodal points shaded green in the figure below.
\begin{center}
    \begin{tikzpicture}[scale=1]

    \coordinate (v1) at (0.9239795325,0.3826834324);
    \coordinate (v2) at (0.3826834324,0.9239795325);
    \coordinate (v3) at (-0.3826834324,0.9239795325);
    \coordinate (v4) at (-0.9239795325,0.3826834324);
    \coordinate (v5) at (-0.9239795325,-0.3826834324);
    \coordinate (v6) at (-0.3826834324,-0.9239795325);
    \coordinate (v7) at (0.3826834324,-0.9239795325);
    \coordinate (v8) at (0.9239795325,-0.3826834324);
    
    \begin{scope}[thick]
    \draw[blue] (v1)--(v2);
    \draw[blue] (v2)--(v3);
    \draw[blue] (v3)--(v4);
    \draw[blue] (v4)--(v5);
    \draw[blue] (v5)--(v6);
    \draw[blue] (v6)--(v7);
    \draw[blue] (v7)--(v8);
    \draw[blue] (v8)--(v1);
    \draw[red] (v1)--(v3);
    \draw[blue] (v1)--(v4);
    \draw[red] (v1)--(v5);
    \draw[blue] (v5)--(v8);
    \draw[red] (v6)--(v8);
    \end{scope}
    
    \begin{scope}[dashed]
    \draw[red] (v1)--(v7);
    \draw[red] (v2)--(v6);
    \draw[blue] (v2)--(v7);
    \draw[red] (v3)--(v5);
    \draw[blue] (v3)--(v6);
    \end{scope}
    
    \node[circle,draw,fill=green,scale=0.4] (c1) at (v1){$E$};
    \node[circle,draw,fill=white,scale=0.4] (c2) at (v2){$D$};
    \node[circle,draw,fill=white,scale=0.4] (c3) at (v3){$C$};
    \node[circle,draw,fill=white,scale=0.4] (c4) at (v4){$B$};
    \node[circle,draw,fill=white,scale=0.4] (c5) at (v5){$A$};
    \node[circle,draw,fill=green,scale=0.4] (c6) at (v6){$H$};
    \node[circle,draw,fill=white,scale=0.4] (c7) at (v7){$G$};
    \node[circle,draw,fill=white,scale=0.4] (c7) at (v8){$F$};
    
    \coordinate (L1) at (0,{-sqrt(3)/2});
    \node[yshift=-7mm] at (L1) {Class 12};    
    
\end{tikzpicture}
\end{center}

The segment $\overline{HE}$ is orthogonal to the plane containing the points $A,C$ and $G$.  Letting $|\overline{AE}|=a$ and $|\overline{AH}|=b$, we have $a^2+b^2=4$.  Without loss of generality, we may assume that $B=e_3$ and that the plane containing the points $A,C$ and $E$ is given by $z=h$, where $h\in(-1,1)$.  This allows us to write $b=\sqrt{2(1-h)}$ and $a=\sqrt{3(1-h^2)}$.  Hence we have
\[
0=3(1-h^2)+2(1-h)-4=-3h^2-2h+1=(1+h)(1-3h),
\]
whose only feasible solution is $h=1/3$, yielding  $a=\sqrt{8/3}$ and $b=\sqrt{4/3}$.  However, this solution is extraneous. To see this, note that the plane containing $D,F$ and $H$ is parallel to the plane containing $A,C$ and $E$.  This forces $\triangle(DFH)$ to be equilateral, so that $|\overline{DF}|=a$.  Since $G$ is equidistant from $D,F$ and $H$, it sits at the apex of the cap cut off by the plane containing these points (which has equation $z=-1/3$). Hence, $\overline{BG}$ is a diameter and  $|\overline{EG}|=|\overline{AG}|=|\overline{CG}|=a$.  Similarly, we obtain  $|\overline{BD}|=|\overline{BF}|=|\overline{BH}|=a$.  The vertices $B,E,G$ and $H$ lie on the same great circle and form a rectangle, with $C$ and $D$ lying in one hemisphere and $A$ and $F$ in the other.  Since the triangles $\triangle(EFG), \triangle(DEG), \triangle(ABH)$ and $\triangle(BCH)$ are congruent, $\overline{AC}$ crosses the $BEGH$ plane, and $b=|\overline{AF}|=|\overline{CD}|$, we cannot simultaneously preserve the convexity on the $D,E,F,G$ side and on the $A,B,C,H$ side.    


\subsubsection{Class 13}  

A member of this class is a hexagonal bipyramid, which is the convex hull of a hexagon and two apexes on opposite sides, so that the segment joining the apexes intersects the relative interior of the hexagon. The  surface area of an inscribed hexagonal bipyramid $P$ satisfies the bound 
\[
S(P) \leq 12\sqrt{1+\cos^2\frac{\pi}{6}}\sin\frac{\pi}{6} = 3\sqrt{7}=7.93725\ldots
\]
with equality if and only if $P$ is the convex hull of a regular hexagon in the equator and the north and south poles (see \cite[Cor. 2]{DHL}). 

\subsubsection{Class 14}  
The majority of colorings are ruled out by Defect B, leaving only 3 plausible colorings (shown below) which do not exhibit any of the coloring defects described.  

\begin{center}
    \begin{tikzpicture}[scale=1]
    
    \coordinate (v1) at (0.9239795325,0.3826834324);
    \coordinate (v2) at (0.3826834324,0.9239795325);
    \coordinate (v3) at (-0.3826834324,0.9239795325);
    \coordinate (v4) at (-0.9239795325,0.3826834324);
    \coordinate (v5) at (-0.9239795325,-0.3826834324);
    \coordinate (v6) at (-0.3826834324,-0.9239795325);
    \coordinate (v7) at (0.3826834324,-0.9239795325);
    \coordinate (v8) at (0.9239795325,-0.3826834324); 
        
    \begin{scope}[thick]
    \draw[blue] (v1)--(v2);
    \draw[red] (v2)--(v3);
    \draw[blue] (v3)--(v4);
    \draw[blue] (v4)--(v5);
    \draw[blue] (v5)--(v6);
    \draw[red] (v6)--(v7);
    \draw[blue] (v7)--(v8);
    \draw[blue] (v8)--(v1);
    \draw[blue] (v1)--(v3);
    \draw[red] (v1)--(v4);
    \draw[blue] (v1)--(v5);
    \draw[blue] (v5)--(v7);
    \draw[red] (v5)--(v8);
    \end{scope}
    
       \begin{scope}[dashed]
    \draw[blue] (v2)--(v7);
    \draw[red] (v2)--(v8);
    \draw[blue] (v3)--(v6);
    \draw[blue] (v3)--(v7);
    \draw[red] (v4)--(v6);
    \end{scope}
    
     \node[circle,draw,fill=green,scale=0.4] (c1) at (v1){$E$};
    \node[circle,draw,fill=white,scale=0.4] (c2) at (v2){$D$};
    \node[circle,draw,fill=yellow,scale=0.4] (c3) at (v3){$C$};
    \node[circle,draw,fill=white,scale=0.4] (c4) at (v4){$B$};
    \node[circle,draw,fill=yellow,scale=0.4] (c5) at (v5){$A$};
    \node[circle,draw,fill=white,scale=0.4] (c6) at (v6){$H$};
    \node[circle,draw,fill=green,scale=0.4] (c7) at (v7){$G$};
    \node[circle,draw,fill=white,scale=0.4] (c7) at (v8){$F$};
    
    \coordinate (L1) at (0,{-sqrt(3)/2});
    \node[yshift=-7mm] at (L1) {Class 14 (i)};    
    
\end{tikzpicture}
\begin{tikzpicture}[scale=1]
    
    \coordinate (v1) at (0.9239795325,0.3826834324);
    \coordinate (v2) at (0.3826834324,0.9239795325);
    \coordinate (v3) at (-0.3826834324,0.9239795325);
    \coordinate (v4) at (-0.9239795325,0.3826834324);
    \coordinate (v5) at (-0.9239795325,-0.3826834324);
    \coordinate (v6) at (-0.3826834324,-0.9239795325);
    \coordinate (v7) at (0.3826834324,-0.9239795325);
    \coordinate (v8) at (0.9239795325,-0.3826834324); 
        
    \begin{scope}[thick]
    \draw[red] (v1)--(v2);
    \draw[blue] (v2)--(v3);
    \draw[red] (v3)--(v4);
    \draw[blue] (v4)--(v5);
    \draw[red] (v5)--(v6);
    \draw[blue] (v6)--(v7);
    \draw[red] (v7)--(v8);
    \draw[blue] (v8)--(v1);
    \draw[blue] (v1)--(v3);
    \draw[blue] (v1)--(v4);
    \draw[red] (v1)--(v5);
    \draw[blue] (v5)--(v7);
    \draw[blue] (v5)--(v8);
    \end{scope}
    
       \begin{scope}[dashed]
    \draw[blue] (v2)--(v7);
    \draw[blue] (v2)--(v8);
    \draw[blue] (v3)--(v6);
    \draw[red] (v3)--(v7);
    \draw[blue] (v4)--(v6);
    \end{scope}
    
     \node[circle,draw,fill=white,scale=0.4] (c1) at (v1){$E$};
    \node[circle,draw,fill=white,scale=0.4] (c2) at (v2){$D$};
    \node[circle,draw,fill=white,scale=0.4] (c3) at (v3){$C$};
    \node[circle,draw,fill=white,scale=0.4] (c4) at (v4){$B$};
    \node[circle,draw,fill=white,scale=0.4] (c5) at (v5){$A$};
    \node[circle,draw,fill=white,scale=0.4] (c6) at (v6){$H$};
    \node[circle,draw,fill=white,scale=0.4] (c7) at (v7){$G$};
    \node[circle,draw,fill=white,scale=0.4] (c7) at (v8){$F$};
    
    \coordinate (L1) at (0,{-sqrt(3)/2});
    \node[yshift=-7mm] at (L1) {Class 14 (ii)};    
    
\end{tikzpicture}
\begin{tikzpicture}[scale=1]
    
    \coordinate (v1) at (0.9239795325,0.3826834324);
    \coordinate (v2) at (0.3826834324,0.9239795325);
    \coordinate (v3) at (-0.3826834324,0.9239795325);
    \coordinate (v4) at (-0.9239795325,0.3826834324);
    \coordinate (v5) at (-0.9239795325,-0.3826834324);
    \coordinate (v6) at (-0.3826834324,-0.9239795325);
    \coordinate (v7) at (0.3826834324,-0.9239795325);
    \coordinate (v8) at (0.9239795325,-0.3826834324); 
        
    \begin{scope}[thick]
    \draw[blue] (v1)--(v2);
    \draw[blue] (v2)--(v3);
    \draw[blue] (v3)--(v4);
    \draw[blue] (v4)--(v5);
    \draw[blue] (v5)--(v6);
    \draw[blue] (v6)--(v7);
    \draw[blue] (v7)--(v8);
    \draw[blue] (v8)--(v1);
    \draw[red] (v1)--(v3);
    \draw[blue] (v1)--(v4);
    \draw[red] (v1)--(v5);
    \draw[red] (v5)--(v7);
    \draw[blue] (v5)--(v8);
    \end{scope}
    
       \begin{scope}[dashed]
    \draw[blue] (v2)--(v7);
    \draw[red] (v2)--(v8);
    \draw[blue] (v3)--(v6);
    \draw[red] (v3)--(v7);
    \draw[red] (v4)--(v6);
    \end{scope}
    
     \node[circle,draw,fill=white,scale=0.4] (c1) at (v1){$E$};
    \node[circle,draw,fill=white,scale=0.4] (c2) at (v2){$D$};
    \node[circle,draw,fill=white,scale=0.4] (c3) at (v3){$C$};
    \node[circle,draw,fill=white,scale=0.4] (c4) at (v4){$B$};
    \node[circle,draw,fill=white,scale=0.4] (c5) at (v5){$A$};
    \node[circle,draw,fill=white,scale=0.4] (c6) at (v6){$H$};
    \node[circle,draw,fill=white,scale=0.4] (c7) at (v7){$G$};
    \node[circle,draw,fill=white,scale=0.4] (c7) at (v8){$F$};
    
    \coordinate (L1) at (0,{-sqrt(3)/2});
    \node[yshift=-7mm] at (L1) {Class 14 (iii)};   
\end{tikzpicture}
\end{center}

\subsubsection{Class 14 (i)}  Lemma \ref{rule-out-lemma} provides two pairs of antipodal vertices,  shaded green and yellow, with diameters $\overline{AC}$ and $\overline{EG}$.  This tells us the blue length, $b=\sqrt{2}$.   The points $A,C,D$ and $F$ are coplanar because they are equidistant from the point $E$.  Moreover, $E$ sits at the apex of the cap cut off by this plane.  The coloring forces congruent central angles between the pairs of vertices $A$ and $F$, $F$ and $D$, and $D$ and $C$.  Since $\overline{AC}$ is a diameter, the shared angle is $\pi/3$, which means that $a=1$.  Thus, the surface area of this configuration equals
\[
12\dfrac{\sqrt{7}}{4}=3\sqrt{7}\approx 7.94.
\]

\subsubsection{Class 14 (ii)}  
The vertex $B$ sits at the apex of the cap cut off by the plane containing vertices $A,E$ and $H$, since these are equidistant from $B$.  Notice that $|\overline{AH}|=|\overline{AE}|=a$ and $|\overline{CE}|=|\overline{CH}|=b$.  This forces the vertices $A,B$ and $C$ to lie in the same great circle.  Consider the sphere centered at $C$ with radius $a$.  It intersects the unit sphere in a circle which contains both $B$ and $G$.  Similarly, the sphere of radius $b$ centered at $A$ intersects the unit sphere in a circle containing $B$ and $G$.  Since $A,B$ and $C$ lie on the same great circle, the only way for this to be true is if $\overline{AC}$ is a diameter.  Similarly, $\overline{EG}$ is a diameter.  This provides $4=a^2+b^2$, as well as $a=|\overline{AD}|=|\overline{CF}|=|\overline{EH}|=|\overline{BG}|$.  The only way that this is possible is for the two sets of vertices $A,D,E $ and $H$ to be coplanar, as well as the vertices $B,C,F$ and $G$.  This forces the polytope to be a parallelepiped, with two rhomboid facets and four rectangular facets.  However, such a configuration is impossible since some of the edges are contained in the relative interiors of some of the facets.  

\subsubsection{Class 14 (iii)} 

Let the vertices be labeled as follows, where the red edges have length $a$ and the blue edges have length $b$:

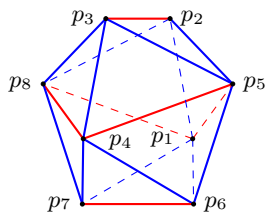
\begin{center}
\tdplotsetmaincoords{90}{60}
\def\r{1}
\begin{tikzpicture}[scale=1.5,line join=bevel, tdplot_main_coords]
    \coordinate (O) at (0,0,0);

\coordinate (P1) at ({sin(3*34.6927133)},0,{cos(3*34.6927133)});
\coordinate (P2) at ({sin(34.6927133)},0,{cos(34.6927133)});
\coordinate (P3) at ({-sin(34.6927133)},0,{cos(34.6927133)});
\coordinate (P4) at ({-sin(3*34.6927133)},0,{cos(3*34.6927133)});
\coordinate (P5) at (0,{-sin(3*34.6927133 )},{-cos(3*34.6927133)});
\coordinate (P6) at (0,{-sin(34.6927133 )},{-cos(34.6927133)});
\coordinate (P7) at (0,{sin(34.6927133 )},{-cos(34.6927133)});
\coordinate (P8) at (0,{sin(3*34.6927133 )},{-cos(3*34.6927133)});

\begin{scope}[thick]
    \draw[red] (P2) -- (P3) node[below,pos=0.8,xshift=4]{};
      \draw[blue] (P3)--(P5)node[below,pos=0.3,xshift=-4]{};
      \draw[blue] (P3)--(P4) node[below,pos=0.4,xshift=-3]{};
      \draw[blue] (P3)--(P8) node[below,pos=0.6,xshift=3]{};
      \draw[red] (P4)--(P5) node[below,pos=0.6,xshift=3]{};
       \draw[blue] (P4)--(P6) node[below,pos=0.6,xshift=3]{};
        \draw[blue] (P4)--(P7) node[below,pos=0.6,xshift=3]{};
         \draw[red] (P4)--(P8) node[below,pos=0.6,xshift=3]{};
          \draw[blue] (P5)--(P6) node[below,pos=0.6,xshift=3]{};
           \draw[red] (P6)--(P7) node[below,pos=0.6,xshift=3]{};
            \draw[blue] (P7)--(P8) node[below,pos=0.6,xshift=3]{};
            \draw[blue] (P2) -- (P8) node[below,pos=0.6]{};
\end{scope} 

\begin{scope}[dashed] 
    \draw[blue] (P1) -- (P2) node[below,pos=0.6]{};
    \draw[red] (P1) -- (P5) node[below,pos=0.6]{};
    \draw[blue] (P1) -- (P6) node[below,pos=0.6]{};
    \draw[blue] (P1) -- (P7) node[below,pos=0.6]{};
    \draw[red] (P1) -- (P8) node[below,pos=0.6]{};
     \draw[blue] (P2)--(P5) node[below,pos=0.3,xshift=4]{};
\end{scope}

\filldraw[black] (P1) circle (0.5pt) node[anchor=east,xshift=-1mm]{\scriptsize $p_1$};
\filldraw[black] (P2) circle (0.5pt) node[anchor=west]{\scriptsize $p_2$};
\filldraw[black] (P3) circle (0.5pt) node[anchor=east]{\scriptsize $p_3$};
\filldraw[black] (P4) circle (0.5pt) node[anchor=west,xshift=2mm,yshift=-0.5mm]{\scriptsize $p_4$};
\filldraw[black] (P5) circle (0.5pt) node[anchor=east]{\scriptsize $p_8$};
\filldraw[black] (P6) circle (0.5pt) node[anchor=east]{\scriptsize $p_7$};
\filldraw[black] (P7) circle (0.5pt) node[anchor=west]{\scriptsize $p_6$};
\filldraw[black] (P8) circle (0.5pt) node[anchor=west]{\scriptsize $p_5$};
\end{tikzpicture}
\end{center}

The edge $[p_2,p_3]$ is a diameter of a circle $C$ of radius $a/2$; without loss of generality, $[p_2,p_3]$ is parallel to $e_2$ and $h=\dist(o,[p_2,p_3])$. Thus $C$ has radius $R(h)=\sqrt{1-h^2}$, so $a(h)=2R(h)$, and  $C=C(h)=\mathbb{S}^2\cap\left(e_3^\perp+he_3\right)$. Notice that $p_5,p_6,p_7$ and $p_8$ lie on the same great circle and hence are coplanar. This plane is orthogonal to $[p_2,p_3]$ and intersects its midpoint. Hence $[p_6,p_7]$ lies in  $-C(h)$ and is parallel to $e_1$ (so the red edges  $[p_2,p_3]$ and $[p_6,p_7]$ have orthogonal directions). Thus, $p_2=(0,R(h),h), p_3=(0,-R(h),h), p_6=(R(h),0,-h)$ and $p_7=(-R(h),0,-h)$. Since $\dist(p_1,p_2)=\dist(p_1,p_6)$, setting $p_1=(0,y_1,z_1)$ (where $z_1<0$), we have
\[
(y_1-R(h))^2+(z_1-h)^2=R(h)^2+y_1^2+(z_1+h)^2.
\]
This implies $y_1=-\frac{2h}{R(h)}z_1$. Substituting this into $y_1^2+z_1^2=1$, we obtain $z_1^2=\frac{1-h^2}{1+3h^2}$. Hence $z_1=-\sqrt{\frac{1-h^2}{1+3h^2}}$ since $z_1<0$. This yields $y_1=\frac{2h}{\sqrt{1+3h^2}}$, so $p_1=\left(0,\frac{2h}{\sqrt{1+3h^2}},-\sqrt{\frac{1-h^2}{1+3h^2}}\right)$. Setting $p_5=(x_1,0,z_2)$ with $z_2> 0$, the equation $\dist(p_3,p_5)=\dist(p_5,p_6)$ yields $p_5=\left(\frac{2h}{\sqrt{1+3h^2}},0,\sqrt{\frac{1-h^2}{1+3h^2}}\right)$.  Now $a(h)^2=\dist(p_1,p_5)^2$, so  $4R(h)^2=\dist(p_1,p_5)^2$, which is equivalent to $1-h^2=\frac{h^2+1}{1+3h^2}$. The  solutions to this equation are $h=0$ and $h=\pm 1/\sqrt{3}$; by construction, only $h^*=1/\sqrt{3}$ is possible. Therefore, the equation
$b(h^*)^2=\dist(p_1(h^*),p_2(h^*))^2$ yields $b(h^*)=1$ and $a(h^*)=2R(h^*)=2\sqrt{2}/\sqrt{3}$. Thus, by the congruent facets assumption, the area of each facet equals
\[
A(h^*)=\frac{1}{2}a(h^*)\sqrt{b(h^*)^2-\frac{a(h^*)^2}{4}}=\frac{\sqrt{2}}{3},
\]
so the total surface area equals $12A(h^*)=4\sqrt{2}\approx 5.657$.

\subsection{The cases of equilateral facets}\label{equil-sec}

There are eight convex \emph{deltahedra}, which are polytopes with congruent equilateral facets \cite{F-W-1947}.  The only one with 8 vertices is the  \emph{dodecadeltahedron} (also called a \emph{snub disphenoid}), which is a member of Class 14 \cite{F-W-1947}. We show that it  is not inscribable. Let $q=0.169\ldots$ denote the positive real root of $2x^3+11x^2+4x-1$. The coordinates of a dodecadeltahedron with centroid at the origin are given by $(\pm t,r,0)$, $(0,-r,\pm t)$, $(\pm 1,s,0)$ and $(0,s,\pm 1)$ where $r=\sqrt{q}$, $s=\sqrt{\frac{1-q}{2q}}$ and $t=\sqrt{2-2q}$ (see \cite[p. 427]{Sloane-etal-1995}). With the centroid fixed, these points are unique up to rotations and dilations. The points $(\pm t, r, 0)$ and $(0,-r,\pm t)$ have distance $\sqrt{r^2+t^2}$ from the origin, while the points $(\pm 1,s,0)$ and $(0,s,\pm 1)$ have distance $\sqrt{1+s^2}$ from the origin. Since $r^2+t^2\neq 1+s^2$, we conclude that these points cannot all  be scaled by the same factor to lie on the unit sphere. This property is invariant under translations, so if a dodecadeltahedron has centroid not at the origin, we can apply rigid motions so that it has the above coordinates. Thus, Class 14 does not admit a member inscribed in $\mathbb{S}^2$ with congruent equilateral facets.  \qed

\subsection{Conclusion of the proof of Theorem \ref{mainThm}}

Comparing the results above for the 14 combinatorial classes, we see that the maximum surface area polytope in $\mathcal{M}_8$ is the member of Class 10 whose surface area equals 8. This completes the proof of Theorem \ref{mainThm}. \qed

\section{Further remarks}

There are 50 nonisomorphic combinatorial types of simplicial convex polytopes with 9 vertices. The methods used in this article can be used to greatly reduce the number of possible colorings to just a handful of cases. Each of those cases would then need to be analyzed individually using geometric methods to determine the surface area maximizer in $\mathcal{M}_9$. In general, the same  method from this paper will work for finding the surface area maximizer in $\mathcal{M}_K$ for any integer $K\geq 4$. 

\section*{Acknowledgments}

The authors would like to thank the Perspectives on Research In Science \&  Mathematics (PRISM) program at Longwood University for its  support in funding this research project.

\bibliographystyle{plain}
\bibliography{main}

\begin{thebibliography}{10}

\bibitem{akkiraju}
N.~Akkiraju.
\newblock Approximating spheres and sphere patches.
\newblock {\em Computer Aided Geometric Design}, 15:739--756, 1998.

\bibitem{BHK}
F.~Besau, S.~Hoehner, and G.~Kur.
\newblock Intrinsic and {D}ual {V}olume {D}eviations of {C}onvex {B}odies and
  {P}olytopes.
\newblock {\em International Mathematics Research Notices}, 22:17456--17513,
  2021.

\bibitem{BrittonDunitz1973}
D.~Britton and J.~D. Dunitz.
\newblock A {C}omplete {C}atalogue of {P}olyhedra with {E}ight or {F}ewer
  {V}ertices.
\newblock {\em Acta Crystallographica Series A}, 29(4):367--371, 1973.

\bibitem{DHL}
J.~Donahue, S.~Hoehner, and B.~Li.
\newblock {The maximum surface area polyhedron with five vertices inscribed in
  the sphere $\mathbb{S}^2$}.
\newblock {\em Acta Crystallographica Section A}, 77(1):67--74, 2021.

\bibitem{Toth-RegularFigures}
L.~Fejes~T\'oth.
\newblock {\em Regular Figures}, volume~48 of {\em International Series of
  Monographs on Pure \& Applied Mathematics}.
\newblock Permagon Press, 1964.

\bibitem{F-W-1947}
H.~Freudenthal and B.~L. van~der Waerden.
\newblock Over een bewering van {E}uclides.
\newblock {\em Simon Stevin}, 25:115--128, 1947.

\bibitem{GruberIn}
P.~M. Gruber.
\newblock Volume approximation of convex bodies by inscribed polytopes.
\newblock {\em Mathematische Annalen}, 281(2):229--245, 1988.

\bibitem{GruberOut}
P.~M. Gruber.
\newblock Volume approximation of convex bodies by circumscribed polytopes.
\newblock In {\em Applied Geometry and Discrete Mathematics: The Victor Klee
  Festschrift}, volume~4, pages 309--317. American Mathematical Society, 1991.

\bibitem{HK-DCG}
S.~Hoehner and G.~Kur.
\newblock A {C}oncentration {I}nequality for {R}andom {P}olytopes,
  {D}irichlet-{V}oronoi {T}iling {N}umbers and the {G}eometric {B}alls and
  {B}ins {P}roblem.
\newblock {\em Discrete \& Computational Geometry}, 65(3):730--763, 2021.

\bibitem{Hoehner-Ledford-2022}
S.~Hoehner and J.~Ledford.
\newblock Extremal arrangements of points on the sphere for weighted
  cone-volume functionals.
\newblock {\em arXiv: 2205.09096}, 2022.

\bibitem{kazakov-thesis}
M.~Kazakov.
\newblock The {S}tructure of the {R}eal {N}umerical {R}ange and the {S}urface
  {A}rea {Q}uantum {E}ntanglement {M}easure.
\newblock Master's thesis, The University of Guelph, 12 2018.

\bibitem{Padrol2016}
A.~Padrol and G.~M. Ziegler.
\newblock {\em Six Topics on Inscribable Polytopes}, pages 407--419.
\newblock Springer Berlin Heidelberg, Berlin, Heidelberg, 2016.

\bibitem{Sloane-etal-1995}
N.~J.~A. Sloane, R.~H. Hardin, T.~D.~S. Duff, and J.~H. Conway.
\newblock Minimal-{E}nergy {C}lusters of {H}ard {S}pheres.
\newblock {\em Discrete \& Computational Geometry}, 14:237--259, 1995.

\bibitem{zador1982asymptotic}
Paul Zador.
\newblock Asymptotic quantization error of continuous signals and the
  quantization dimension.
\newblock {\em IEEE Transactions on Information Theory}, 28(2):139--149, 1982.

\end{thebibliography}


\vspace{3mm}

\noindent {\sc Department of Mathematics \& Computer Science, Longwood University, U.S.A.}

\noindent {\it E-mail addresses:} {\tt nicolas.freeman@live.longwood.edu;  hoehnersd@longwood.edu; \\ \noindent ledfordjp@longwood.edu;  david.pack@live.longwood.edu;\\  \noindent brandon.walters@live.longwood.edu}

\end{document}